\documentclass{amsart}

\usepackage[normalem]{ulem}

\usepackage[noadjust]{cite}
\usepackage{color}
\usepackage{enumitem}
\usepackage{graphicx}

\usepackage[english]{babel}
\usepackage[dvipsnames]{xcolor}
\usepackage{hyperref}
\usepackage{cleveref}

\newcommand\myshade{85}
\colorlet{mylinkcolor}{violet}
\colorlet{mycitecolor}{YellowOrange}
\colorlet{myurlcolor}{Aquamarine}

\hypersetup{
	linkcolor  = mylinkcolor!\myshade!black,
	citecolor  = mycitecolor!\myshade!black,
	urlcolor   = myurlcolor!\myshade!black,
	colorlinks = true,
}

\usepackage[hyperpageref]{backref}

\usepackage[T1]{fontenc}
\usepackage{lmodern}
\usepackage[activate={true,nocompatibility},final,tracking=true,kerning=true,factor=1100,stretch=10,shrink=10]{microtype}

\usepackage{fixmath}

\usepackage{todonotes}

\newtheorem{Th}{Theorem}[section]

\newtheorem{Lem}[Th]{Lemma}
\newtheorem{Cor}[Th]{Corollary}

\theoremstyle{remark}
\newtheorem{Rem}[Th]{Remark}
\newtheorem{Ex}[Th]{Example}

\newcommand{\vp}{\varphi}

\newcommand{\eps}{\varepsilon}

\newcommand{\R}{\mathbb{R}}

\newcommand{\Z}{\mathbb{Z}}

\newcommand{\cC}{{\mathcal C}}
\newcommand{\cD}{{\mathcal D}}
\newcommand{\cE}{{\mathcal E}}
\newcommand{\cF}{{\mathcal F}}

\newcommand{\cH}{{\mathcal H}}
\newcommand{\cI}{{\mathcal I}}

\newcommand{\cN}{{\mathcal N}}

\newcommand{\al}{\alpha}

\newcommand{\Ga}{\Gamma}

\newcommand{\weakto}{\rightharpoonup}

\numberwithin{equation}{section}

\DeclareMathOperator*{\essinf}{ess\,inf}

\newcommand{\supp}{\mathrm{supp}\,}

\newcommand{\loc}{\mathrm{loc}}
\newcommand{\per}{\mathrm{per}}

\begin{document}

\nocite{*}

\title[Semirelativistic Choquard equations]{Semirelativistic Choquard
  equations with singular potentials and general nonlinearities
  arising from Hartree-Fock theory}

\author{Federico Bernini}
\address{Dipartimento di Matematica e Applicazioni, Università
  degli Studi di Milano-Bicocca, via Roberto Cozzi 55, I-20125,
  Milano, Italy}
\email{f.bernini2@campus.unimib.it}

\author{Bartosz Bieganowski} \address{Institute of Mathematics, Polish
  Academy of Sciences, ul. \'Sniadeckich 8, 00-656 Warsaw, Poland and
  Nicolaus Copernicus University, Faculty of Mathematics and Computer
  Science, ul. Chopina 12/18, 87-100 Toru\'n, Poland}
\email{bbieganowski@impan.pl, bartoszb@mat.umk.pl}

\author{Simone Secchi} \address{Dipartimento di Matematica e
  Applicazioni, Università degli Studi di Milano-Bicocca, via Roberto
  Cozzi 55, I-20125, Milano, Italy} \email{simone.secchi@unimib.it}
\date{\today}

\begin{abstract} 
We are interested in the general Choquard equation
\begin{multline*}
  \sqrt{\strut -\Delta + m^2} \ u - mu + V(x)u - \frac{\mu}{|x|} u
  = \\
  \left( \int_{\R^N} \frac{F(y,u(y))}{|x-y|^{N-\alpha}} \, dy \right)
  f(x,u) - K (x) |u|^{q-2}u
\end{multline*}
under suitable assumptions on the bounded potential \(V\) and on the
nonlinearity \(f\). Our analysis extends recent results by the second
and third author on the problem with $\mu = 0$ and pure-power
nonlinearity $f(x,u)=|u|^{p-2}u$. We show that, under appropriate
assumptions on the potential, whether the ground state does exist or
not. Finally, we study the asymptotic behaviour of ground states as
$\mu \to 0^+$.
%
%
%
\end{abstract}

\maketitle

\tableofcontents

\section{Introduction}
\setcounter{section}{1}

In a model for an atom with \(n\) electrons and nuclear charge \(Z\), the kinetic energy of the electrons is described by the expression
\begin{gather*}
\sqrt{(|\mathbf{p}|c)^2 + (mc^2)^2} - mc^2.
\end{gather*}
This model takes into account some relativistic effects, and gives
rise to a Hamiltonian of the form
\begin{gather*}
\mathcal{H} = \sum_{j=1}^n \left\{ \sqrt{-\alpha^{-2} \Delta_j + \alpha^{-4}} - \alpha^{-2} - \frac{Z}{|\mathbf{x}_j|} \right\} + \sum_{1 \leq i < j \leq n} \frac{1}{|\mathbf{x}_i-\mathbf{x}_j|},
\end{gather*}
where \(\alpha\) is Sommerfeld's fine structure constant,
\(\alpha \approx 1/137.036\), see \cite{DASS} and \cite{RGN}. As a
particular case one can consider a one-electron atom like a hydrogen,
for which \(n=1\).

Motivated by this application (with a different scale of unit
measures) we are interested in the general Choquard equation
\begin{multline}
  \sqrt{\strut -\Delta + m^2} \ u - mu + V(x)u - \frac{\mu}{|x|} u
   \\
  =\left( \int_{\R^N} \frac{F(y,u(y))}{|x-y|^{N-\alpha}} \, dy \right)
  f(x,u) - K (x) |u|^{q-2}u \label{eq:1.1}
\end{multline}
where $u \in H^{1/2} (\R^N)$, $V \in L^\infty (\R^N)$ is an external
potential, $m > 0$ and $\mu \neq 0$. The function $f$ is a general
nonlinearity with~$F(x,t):=\int_0^t f(x,s)\, ds$. Observe that in the
equation we consider also the singular part of the potential, i.e.
$ - \frac{\mu}{|x|}$, which is singular at the origin and does not
belong to $L^\infty (\R^N)$ nor $L^k (\R^N)$ for $k \geq 1$. Moreover it does not belog to the Kato's class and cannot be treated as a lower order perturbation of the nonlocal operator $\sqrt{-\Delta + m} -m + V(x)$ in $L^2 (\R^N)$.

Equation \eqref{eq:1.1} arises from the \emph{ansatz}
\begin{equation*}
\Psi(t, x) = e^{it \lambda} u(x)
\end{equation*}
in the time-dependent, pseudorelativistic Hartree problem
\begin{multline}\label{eq:time-dep}
  \mathrm{i} \frac{\partial \Psi}{\partial t} = \left( \sqrt{-\Delta +
      m^2} - m \right) \Psi + \left( V(x) - \frac{\mu}{|x|} +\lambda
  \right) \Psi \\ - \left( I_\alpha * F(x, |\Psi|) \right) f(x,
  |\Psi|) + K(x) |\Psi|^{q-2} \Psi
\end{multline}
or (up to some normalization constant $c_{N, \alpha}$) in the following fractional version of semirelativistic Schr\"odinger-Poisson system
\begin{equation*}
  \left\{ \begin{array}{l}
            \mathrm{i} \frac{\partial \Psi}{\partial t} = \left( \sqrt{-\Delta + m^2} - m \right) \Psi + \left( V(x) - \frac{\mu}{|x|} +\lambda  \right) \Psi  - \Phi f(x, |\Psi|) + K(x) |\Psi|^{q-2} \Psi, \\
            (-\Delta)^{\alpha / 2} \Phi = c_{N, \alpha} F(x, |\Psi|).
\end{array} \right.
\end{equation*}
%
These equations appear in quantum theory for large systems of
self-interacting bosons with prescribed mass~$m > 0$. In particular,
they appear in models in astrophysics describing the evolution of
many-body quantum systems, like boson stars. The external potential
accounts for gravitational fields from other stars. It is also
possible to describe the evolution of other type stars, like white
dwarfs or neutron stars, using the time-dependent equation of the form
\eqref{eq:time-dep}. In particular, in \cite{F1} the collapse of white
dwarfs has been studied via the analysis of existence and blow-up of
solutions to the Hartee and Hartree-Fock equations. See also \cite{ES,
  F2, F3, F4} for more details on these physical models.

In \cite{FL} Fefferman and de la Llave showed how a system governed by
operator $\cH$ can implode: in fact, this is happen for a single
quantized electron attracted to a single nucles of charge $Z$ fixed at
the origin.  In \cite{LY2} Lieb and Yau studied the quantum mechanical
many-body problem where they consider the problem where electrons and
fixed nuclei interact via Coulomb forces with a relativistic kinect
energy: they proved that stability of relatve matters occurs for
suitable values $Z$ and $\al$.  In \cite{LY1}, the same authors
consider operator $\cH$ with $Z=0$, that is electrically neutral
gravitating particles (e.g. fermions or bosons) and they showed that
the ground state of stars can be obtained as the limit  $G$ (the
gravitation constant) goes to zero and $n$ (the number of particles)
goes to infinity. We refer to \cite{Lieb,LS,LT} for further results.

\vspace{1cm}

The stationary Schr\"odinger equation with singular potential of the
form
\begin{equation}\label{schrodinger}
  -\Delta u + V(x) u - \frac{\mu}{|x|^2} u = f(x,u), \quad x \in \R^N
\end{equation}
has been studied in \cite{GM}. In their paper, Guo and Mederski show
that \eqref{schrodinger} admits a ground state solution for
sufficiently small~$\mu >0$ in a strongly indefinite setting by means
of a linking-type argument connected with the classical Hardy
inequality in $H^1 (\R^N)$. They were able to study the strongly
indefinite case, i.e. the situation where the infimum of the spectrum
$\sigma (-\Delta + V(x))$ lies below 0, since the nonlinear part
$I(u) = \int_{\R^N} F(x,u) \, dx$ of the variational functional
$J \colon H^1(\mathbb{R}^N) \to \mathbb{R}$ defined by
\begin{gather*}
J(u) = \frac12 \int_{\R^N} |\nabla u|^2 + V(x) u^2 \, dx -
\frac{\mu}{2} \int_{\R^N} \frac{u^2}{|x|^2} \, dx - \int_{\R^N} F(x,u)
\, dx
\end{gather*}
is nonnegative. Later, this result has been extended in \cite{B} in the
fractional setting for positive potentials $V$, but with sign-changing
nonlinearities. The approach in \cite{B} is based on the Nehari
manifold technique, combined with recent results on the fractional
Hardy inequality.

\vspace{1cm}

Turning to our work, we study the Choquard equation \eqref{eq:1.1}
where the classical laplacian $-\Delta$ is replaced by the nonlocal
operator $\sqrt{-\Delta + m^2}$, which is known in the literature as
the semirelativistic Schr\"odinger operator. The second difficulty is
that the nonlinear part is nonlocal and has additional power-type
term, which makes the right hand side to change sign. As a consequence
we will study only positive potentials $V$.

The problem \eqref{eq:1.1} with $\mu = 0$ and $K \equiv 0$ has been
widely studied with the pure-power type nonlinearity $f$
(see~\cite{MVS1}) and also with general nonlinearity (see~\cite{BVS,
  MVS2}). The case $\mu = 0$ with $K \not\equiv 0$ and the pure-power
nonlinearity $f$ has been studied by the second and third author in
\cite{BS}. See also \cite{CS1, CS2, CZN1, CZN2, dASS,
    Le, MZ, S, SCV} and references therein.

\vspace{1cm}
  
We consider the following relation between numbers $p$, $q$, $\alpha$ and
the dimension $N$.
\begin{enumerate}
\item[(N)] $N \geq 2$, $(N-1)p - N < \alpha < N$,
  $2 <q < \min \{2p, 2N / (N-1)\}$ and $p > 2$.
\end{enumerate}
Note that, in particular, $p < \frac{2N}{N-1}$, so that the growth
parameter $p$ is smaller than the \emph{critical Sobolev exponent} for
the space $H^{1/2}(\R^N)$. In this sense we consider a nonlinearity
$f$ with subcritical growth, see assumption~(F1) below.

The following are our assumptions on the potential function $V$:
\begin{enumerate}
\item[(V1)] $V = V_p + V_l$, where $V_p \in L^\infty (\R^N)$ is
  $\mathbb{Z}^N$-periodic and 
  $V_l \in L^\infty (\R^N) \cap L^N (\R^N)$ satisfies
\begin{align*}
  \lim_{|x|\to+\infty} V_l (x) = 0.
\end{align*}
\item[(V2)]$\essinf_{x \in \R^N} V(x) > m$.
\end{enumerate}

Conditions (V1) and (V2) ensure that the operator
$\sqrt{-\Delta + m^2} + V(x) - m$ is positive definite. In particular,
the quadratic form associated to this operator generates a norm in
$H^{1/2} (\R^N)$ which is equivalent to the standard one. Similar
assumptions were considered in \cite{BS} in the case $\mu = 0$ in the
presence of the pure power nonlinearity $f(x,u) =
|u|^{p-2}u$. However, for $\mu > 0$ it is unclear whether the operator
$\sqrt{-\Delta + m^2} + V(x) - m - \frac{\mu}{|x|}$ is positive
definite, and we can show it only for small values of $\mu$. 
Moreover, the potential $V$ is not necessarily $\Z^N$-periodic, and the
application of Lions' concentration-compactness principle is not
straightforward.

With respect to the nonlinearity \(f\) we assume:
\begin{enumerate}
\item[(F1)] $f\colon \R^N \times \R \rightarrow \R$ is a
  Carath\'eodory function\footnote{We say that $f\colon \R^N \times \R \rightarrow \R$ is a Carath\'eodory function if $f(\cdot, u)$ is measurable for every $u\in\R$ and $f(x, \cdot)$ is continuous for a.e. $x \in \R^N$.}, $\Z^N$-periodic in $x \in \R^N$ and there
  is $C > 0$ such that
\begin{align*}
|f(x,u)| \leq C \left( |u|^{\frac{\alpha}{N}} + |u|^{p-1} \right).
\end{align*}
\item[(F2)] $f(x,u) = o(u)$ as $u \to 0$ uniformly with respect to
  $x$.
\item[(F3)] $F(x,u)/|u|^\frac{q}{2} \to +\infty$ as $|u| \to +\infty$, uniformly
  with respect to $x$, where $F(x,u) = \int_0^u f(x,s) \, ds$ and
\begin{align*}
  F(x,u) \geq 0
\end{align*}
for $u \in \R$ and a.e. $x \in \R^N$.
\item[(F4)] The function $u \mapsto f(x,u)/|u|^{\frac{q-2}{2}}$ is
  non-decreasing on each half-line $(-\infty, 0)$ and $(0,+\infty)$.
\end{enumerate}
  
Finally,
\begin{enumerate}
\item[(K)] $K \in L^\infty(\R^N)$ is $\Z^N$-periodic and
  non-negative.
\end{enumerate}

\begin{Rem}
Assumptions (F3) and (F4) imply that
\begin{align}\label{AR-q}
0 \leq \dfrac{q}{2} F(x,u) \leq f(x,u)u
\end{align}
for almost every $x \in \R^N$ and $u \in \R$, which is a weaker
variant of the well-known Ambrosetti-Rabinowitz condition. It is also
classical to check that conditions (F1), (F2) and (F3) imply that for any
$\eps > 0$ there is $C_\eps > 0$ such that
\begin{equation} \label{F-eps}
  F(x,u) \leq \eps|u|^2 + C_\eps |u|^p
\end{equation}
\end{Rem}
\begin{Ex}
  One can easily check that the pure power nonlinearity
  $f(x,u) = |u|^{p-2}u$ satisfies (F1)--(F4) as soon as (N) holds
  true.
  \end{Ex}
\begin{Ex}
Consider $f(x,u) = L(x) u \log (1+|u|^{p-2})$, with
$\Z^N$-periodic $L \in L^\infty (\R^N)$, $\inf_{x \in \R^N} L(x) >
0$. It is clear that (F1) and (F2) are satisfied. Note that it follows from (N)
 that
\begin{equation*}
q < \frac{2N}{N-1} \leq 4.
\end{equation*}
Hence, to get (F3) we use the L'H\^{o}pital's rule
\begin{align*}
  \lim_{|u| \to +\infty} \frac{F(x,u)}{|u|^{q/2}} &= \frac{2 L(x)}{q} \lim_{|u| \to +\infty} \frac{u \log (1+|u|^{p-2})}{|u|^{\frac{q}{2}-2}u} = \frac{2 L(x)}{q} \lim_{|u| \to +\infty} \frac{\log (1+|u|^{p-2})}{|u|^{\frac{q-4}{2}}} \\
                                                  &= \frac{2 L(x)}{q} \lim_{|u| \to +\infty} |u|^{\frac{4-q}{2}} \log (1+ |u|^{p-2}) = +\infty.
\end{align*}
To get the inequality $F(x,u) \geq 0$ note that for $u \geq 0$ we have
\begin{equation*}
F(x,u) = L(x) \int_0^u s \ln (1+s^{p-2}) \, ds \geq 0.
\end{equation*}
Note that $f$ is odd in $u$, and therefore $F$ is even in $u$. Hence
$F(x,u) = F(x,-u) \geq 0$ for $u < 0$. To obtain (F4) we note that
$f(x,u)/|u|$ is clearly non-decreasing on $(0, +\infty)$. Moreover
$f(x,u) \geq 0$ on $(0, +\infty)$. Hence
\begin{equation*}
\frac{f(x,u)}{|u|^{\frac{q-2}{2}}} = \frac{f(x,u)}{|u|} |u|^{\frac{4-q}{2}}
\end{equation*}
is non-decreasing on $(0,+\infty)$. We proceed similarly on
$(-\infty, 0)$.
\end{Ex}
\begin{Ex}
Suppose that $\tilde{f}$ satisfy (F1)--(F4) and, for simplicity, does not depend on $x \in \R^N$. It is clear from (F1), (F2), (F4) that $\tilde{f}(u) > 0$ on $(0,+\infty)$. Take $M > 1$ and set, for $u \geq 0$,
\begin{equation*}
f(x, u) =   \begin{cases}
L(x)\tilde{f}(u) &\text{if $u < 1$}, \\
L(x) \tilde{f}(1) u^{\frac{q-2}{2}} &\text{if $1 \leq u \leq M$}, \\
L(x) \frac{M^{\frac{q-2}{2}} \tilde{f}(1)}{\tilde{f}(M)} \tilde{f}(u) & \text{if $u > M$}
\end{cases} 
\end{equation*}
and $f(x,u) = f(x,-u)$ for $u < 0$, where $L \in L^\infty (\R^N)$ is
$\Z^N$-periodic and $\inf_{x \in \R^N} L(x) > 0$. Then we can easily
check that (F1)--(F4) are satisfied. Moreover, on $[1,M]$ the function
$f$ is sublinear, since $\frac{q-2}{2} \leq 1$.
\end{Ex}
\bigskip

We define the semirelativistic operator $\sqrt{-\Delta + m^2}$ by
means of the Fourier transform, given by the symbol
$\sqrt{|\xi|^2 + m^2}$, i.e. for any rapidly decaying function $u$ we
define that $\sqrt{-\Delta + m^2} u$ is defined by
\begin{equation*}
\cF \left(\sqrt{-\Delta + m^2} u \right) = \sqrt{|\xi|^2 + m^2} \cF(u),
\end{equation*}
where $\cF$ is the Fourier transform on $L^2 (\R^N)$.

Our first result shows that equation \eqref{eq:1.1} possesses a
least-energy solution as long as the parameter $\mu$ is sufficiently
small.
\begin{Th}\label{ThMain1}
  Suppose that (N), (V1), (V2), (F1)--(F4), (K) are satisfied. There
  exists $\mu^* > 0$ such that for all $\mu \in (0, \mu^*)$ and any
  $V_{l}$ satisfying
\begin{equation}\label{a1}
V_{l} (x)  < \frac{\mu}{|x|} \mbox{ for a.e. } x \in \R^N \setminus \{0\},
\end{equation}
there is a ground state solution $u \in H^{1/2} (\R^N)$ of
\eqref{eq:1.1}. The constant 
\begin{equation*}
\mu^* = \mu^* (N) := 2 \frac{\Ga \left( \frac{N+1}{4}\right)^2}{\Ga \left( \frac{N-1}{4} \right)^2},
\end{equation*} 
where $\Ga$ denotes the Euler $\Ga$-function, depends only on the
dimension of the space $N$, but is independent of
the potential $V$ or of the nonlinearity $f$.
\end{Th}

\begin{Rem}
We emphasize the fact that we do not require $V_l (x) < 0$ for a.e.
$x \in \R^N$, and indeed $V_{l}$ may be positive in some neighborhood
of the origin.
\end{Rem}

\begin{Ex}
For $N=2$ the constant $\mu^*$ is equal to
\begin{equation*}
  \mu^* (2) = 2 \frac{\Ga \left( \frac{3}{4}\right)^2}{\Ga \left( \frac{1}{4} \right)^2} = 2 \frac{ \left( \frac{\sqrt{2} \pi}{\Ga \left( \frac{1}{4} \right)} \right)^2}{\Ga \left( \frac{1}{4} \right)^2} = \frac{4 \pi^2}{\Ga \left( \frac{1}{4} \right)^4} \approx 0.22847,
\end{equation*}
while for $N=3$ it equals
\begin{equation*}
  \mu^* (3) = 2 \frac{\Ga \left( \frac{4}{4}\right)^2}{\Ga \left( \frac{2}{4} \right)^2} =  \frac{2}{\pi} \approx 0.63662.
\end{equation*}
See Figure~\eqref{fig:1}.
\end{Ex}
\begin{figure} \label{fig:1}
\begin{center}
\includegraphics[width=5cm]{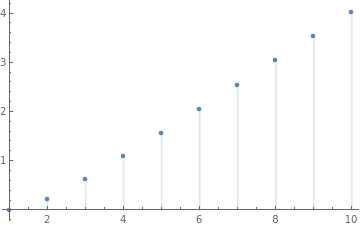}
\end{center}
\caption{Plot of $\mu^*(N)$ as $N$ ranges from $2$ to $10$.}
\end{figure}

Our second result describes a counterpart to Theorem \ref{ThMain1}
when \(\mu < 0\).
\begin{Th}\label{ThMain2}
Suppose that (N), (V1), (F1)--(F4), (K) and
\begin{itemize}
\item[(V2')] $\essinf_{x \in \R^N} V_p (x) > m$
\end{itemize}
are satisfied. If $\mu < 0$ and 
\begin{equation}\label{a2}
V_{l} (x) > \frac{\mu}{|x|} \quad\mbox{for a.e. $x \in \R^N \setminus \{0\}$},
\end{equation}
there are no ground states of \eqref{eq:1.1}.
\end{Th}
Lastly, we focus on the parameter $\mu$, and we prove a compactness
result for ground states as \(\mu \to 0^+\).
\begin{Th}\label{ThMain3}
  Suppose that (N), (V1), (V2), (F1)--(F4), (K) are satisfied and
  $V_l \equiv 0$. Let $\{\mu_n\} \subset (0,\mu^*)$ be a sequence such
  that $\mu_n \to 0^+$. Then for any choice of ground states
  $u_n \in H^{1/2} (\R^N)$ of \eqref{eq:1.1} with $\mu=\mu_n$ there is
  a sequence of translations $\{ z_n \} \subset \Z^N$, such that
\begin{gather*}
u_n (\cdot - z_n) \weakto u_0, \mbox{ up to a subsequence, in } H^{1/2} (\R^N),
\end{gather*}
where $u_0 \in H^{1/2} (\R^N)$ is a ground state solution to
\eqref{eq:1.1} with $\mu = 0$. Moreover $c_n \to c$, where $c_n$ is the energy of $u_n$ and $c$ is the energy of $u$, where the energy is defined by \eqref{E} below.
\end{Th}

\section{Variational setting}

Here and in the sequel $|\cdot|_k$ denotes the usual norm in
$L^k (\R^N)$, where $k \in [1,\infty]$.

The energy functional, associated with \eqref{eq:1.1},
$\cE \colon H^{1/2} (\R^N) \rightarrow \R$ is given by
\begin{multline} \label{E}
\cE(u) := \frac12 \int_{\R^N} \sqrt{|\xi|^2 + m^2} |\hat{u}(\xi)|^2 \, d\xi + \frac12 \int_{\R^N} (V(x)-m) |u(x)|^2 \, dx - \frac{\mu}{2} \int_{\R^N} \frac{|u(x)|^2}{|x|} \, dx \\ - \frac{1}{2} \int_{\R^N \times \R^N} \frac{F(x,u(x))F(y,u(y))}{|x-y|^{N-\alpha}} \, dx \, dy + \frac{1}{q} \int_{\R^N} K(x) |u(x)|^q \, dx.
\end{multline}

\begin{Lem}\label{norm}
The quadratic form
\begin{align*}
	u \mapsto Q(u) := \int_{\R^N} \sqrt{|\xi|^2 + m^2} |\hat{u}(\xi)|^2 \, d\xi +  \int_{\R^N} (V(x)-m) |u(x)|^2 \, dx
\end{align*}
is positive-definite and generates a norm on $H^{1/2} (\R^N)$ that is
equivalent to the standard one. In particular there exist two positive
constants $0 < C (N,m) \leq C (N,m,|V|_\infty)$ such that
\begin{gather*}
C (N,m) \left( [u]^2 + |u|_2^2 \right) \leq Q(u) \leq C (N,m,|V|_\infty) \left( [u]^2 + |u|_2^2 \right) ,
\end{gather*}
where 
\begin{gather*}
[u]^2 := \int_{\R^N \times \R^N} \frac{|u(x)-u(y)|^{2}}{|x-y|^{N+1}} \, dx \, dy
\end{gather*}
is the Gagliardo semi-norm in $H^{1/2} (\R^N)$.
\end{Lem}
\begin{proof}
We note that, by the Plancherel's theorem,
\begin{align*}
  Q(u) &= \int_{\R^N} \sqrt{|\xi|^2 + m^2} |\hat{u}(\xi)|^2 \, d\xi + \int_{\R^N} (V(x)-m) u^2 \, dx \\
       &\leq \int_{\R^N} (|\xi| + m) |\hat{u}(\xi)|^2 \, d\xi + \int_{\R^N} (|V|_\infty-m) u^2 \, dx \\
       &= \int_{\R^N} |\xi| |\hat{u}(\xi)|^2 \, d\xi + \int_{\R^N} |V|_\infty u^2 \, dx \\
       &= \frac12 C\left( N, \frac12 \right) [u]^2 + \int_{\R^N} |V|_\infty u^2 \, dx \\
       &\leq \max \left\{ \frac12 C\left( N, \frac12 \right), |V|_\infty \right\} \left( [u]^2 + |u|_2^2 \right),
\end{align*}
where (see \cite{DNPV})
\begin{gather*}
C\left( N, \frac12 \right) = \left( \int_{\R^N} \frac{1 - \cos \zeta_1}{|\zeta|^{N+1}} \, d\zeta \right)^{-1}.
\end{gather*}
On the other hand we have
\begin{align}
Q(u) &= \int_{\R^N} \sqrt{|\xi|^2 + m^2} |\hat{u}(\xi)|^2 \, d\xi + \int_{\R^N} (V(x)-m) u^2 \, dx \nonumber \\ 
&\geq \int_{\R^N} |\xi| |\hat{u}(\xi)|^2 \, d\xi + \int_{\R^N} \left( \essinf_{\R^N} V - m \right) u^2 \, dx \nonumber \\
&= \frac12 C\left( N, \frac12 \right) [u]^2 + \left( \essinf_{\R^N} V - m \right) |u|_2^2 \nonumber \\
&\geq \min \left\{ \frac12 C\left( N, \frac12 \right), \essinf_{\R^N} V - m \right\} \left( [u]^2 + |u|_2^2 \right).
\label{eq:2.1}
\end{align}
Recall that, see \cite{DNPV},
\begin{gather*}
C\left( N, \frac12 \right)^{-1} = \int_{\R^{N-1}} \frac{1}{(1+|\eta|^2)^{\frac{N+1}{2}}} \, d\eta \cdot \int_{\R} \frac{1 - \cos t}{t^2} \, dt.
\end{gather*}
By contour integration it follows that $\int_{\R} \frac{1 - \cos t}{t^2} \, dt = \pi$, so
\begin{gather*}
C\left( N, \frac12 \right)^{-1} = \pi \int_{\R^{N-1}} \frac{1}{(1+|\eta|^2)^{\frac{N+1}{2}}} \, d\eta.
\end{gather*}
For $N=2$ we have
\begin{gather} \label{eq:C(2,1/2)}
C\left( 2, \frac12 \right)^{-1} = \pi \int_{\R} \frac{1}{(1+|\eta|^2)^{\frac{3}{2}}} \, d\eta =  2 \pi.
\end{gather}

For $N \geq 3$, using polar coordinates we see that
\begin{align*}
C\left( N, \frac12 \right)^{-1} &= \pi \int_{\R^{N-1}} \frac{1}{(1+|\eta|^2)^{\frac{N+1}{2}}} \, d\eta = \pi \omega_{N-2} \int_{0}^{+\infty} \frac{r^{N-2}}{(1+r^2)^{\frac{N+1}{2}}} \, dr \\
&= \pi \omega_{N-2}  \frac{r^{N-1}}{(N-1) (r^2+1)^{\frac{N-1}{2}}} \Bigg|_{r=0}^{r=+\infty} =  \frac{\pi \omega_{N-2}}{N-1}.
\end{align*}
Recalling that $\omega_{N-2} = 2 \pi^{\frac{N-1}{2}} / \Ga \left( \frac{N-1}{2} \right)$ we finally get that
\begin{align} \label{eq:C(N,1/2)}
C\left( N, \frac12 \right)^{-1} = \frac{\pi}{N-1}  \frac{2 \pi^{\frac{N-1}{2}}}{\Ga \left( \frac{N-1}{2} \right)} =  \frac{\pi^{\frac{N+1}{2}}}{ \frac{N-1}{2} \Ga \left( \frac{N-1}{2} \right)} = \frac{\pi^{\frac{N+1}{2}}}{ \Ga \left( \frac{N+1}{2} \right)},
\end{align}
where $\Ga$ is the Euler $\Ga$-function.
\end{proof}

We recall the following variant of the fractional Hardy inequality, adapted to our setting.

\begin{Lem}[{\cite[Theorem 1.1]{FS}}]\label{Hardy}
There exists a constant $C_{N,1/2,1/2} > 0$, depending only on $N \geq 1$, such that for every $u \in H^{1/2} (\R^N)$ there holds
\begin{align*}
	[u]^2 \geq C_{N,1/2,1/2} \int_{\R^N} \frac{|u(x)|^2}{|x|} \, dx,
\end{align*}
where
\[
	C_{N,1/2,1/2}=2\pi^{N/2}\frac{\Ga \left(\frac{N+1}{4} \right)^2 \left|\Ga \left(-\frac{1}{2} \right) \right|}{\Ga \left(\frac{N-1}{4} \right)^2 \Ga \left(\frac{N+1}{2} \right)}
\]
is the sharp constant of the inequality (see \cite{FS}).
\end{Lem}

\begin{Lem}\label{norm-equiv}
There is $\mu^* > 0$ such that for any $0 < \mu < \mu^*$ the quadratic form
\begin{align*}
Q_\mu \colon u \mapsto Q(u)  - \mu \int_{\R^N} \frac{|u(x)|^2}{|x|} \, dx
\end{align*}
is positive-definite and generates a norm on $H^{1/2} (\R^N)$ that is equivalent to the standard one. The constant $\mu^* = \mu^* (N) = 2 \frac{\Gamma \left( \frac{N+1}{4} \right)^2}{\Gamma \left( \frac{N-1}{4} \right)^2}$ depends only on the dimension $N$.
\end{Lem}

\begin{proof}
Note that $Q_\mu (u) \leq Q(u)$ for any $u \in H^{1/2}(\R^N)$. On the other hand, recalling \eqref{eq:2.1} and Lemma~\ref{Hardy},
\begin{align*}
Q_\mu (u) &= Q(u) - \mu \int_{\R^N} \frac{|u(x)|^2}{|x|}\, dx \\
&\geq \int_{\R^N} |\xi| |\hat{u}(\xi)|^2 \, d\xi + \left( \essinf_{\R^N} V -m \right) |u|_2^2 - \mu \int_{\R^N} \frac{|u(x)|^2}{|x|}\, dx \\
&= \frac{1}{2} C \left( N,\frac{1}{2} \right) [u]^2 + \left( \essinf_{\R^N} V -m \right) |u|_2^2 - \mu \int_{\R^N} \frac{|u(x)|^2}{|x|}\, dx \\
&\geq \frac{1}{2} C \left( N,\frac{1}{2} \right) [u]^2 + \left( \essinf_{\R^N} V -m \right) |u|_2^2 - \frac{\mu}{C_{N,1/2,1/2}} [u]^2 \\
&= \left( \frac{1}{2} C\left( N,\frac{1}{2} \right) -\frac{\mu}{C_{N,1/2,1/2}} \right) [u]^2 + \left( \essinf_{\R^N} V-m \right) |u|_2^2 \\
&\geq \min \left\{ \frac{1}{2} C\left( N,\frac{1}{2}\right) - \frac{\mu}{C_{N,1/2,1/2}}, \essinf_{\R^N} V-m \right\} \left( [u]^2+|u|_2^2 \right).
\end{align*}
Recalling \eqref{eq:C(2,1/2)} and \eqref{eq:C(N,1/2)}, the conclusion for $N \geq 3$ follows provided that
\begin{align*}
\mu &< \frac{1}{2} C_{N,1/2,1/2} \cdot C\left( N,\frac{1}{2} \right) \\
&= \frac{1}{2} \cdot 2 \pi^{N/2} \frac{\Gamma \left(\frac{N+1}{4} \right)^2}{\Gamma \left( \frac{N-1}{4} \right)^2} \frac{\left| \Gamma \left( -\frac{1}{2} \right) \right|}{\Gamma \left( \frac{N+1}{2} \right)} \frac{N-1}{\pi} \frac{\Gamma \left( \frac{N-1}{2} \right)}{  2 \pi^{ \frac{N-1}{2} }  } \\
&= (N-1) \frac{\Gamma \left( \frac{N+1}{4} \right)^2}{\Gamma \left( \frac{N-1}{4} \right)^2} \frac{\Gamma \left( \frac{N-1}{2} \right)}{\Gamma \left( \frac{N+1}{2} \right)} = 2 \frac{\Gamma \left( \frac{N+1}{4} \right)^2}{\Gamma \left( \frac{N-1}{4} \right)^2},
\end{align*}
where we have used the fact that
\begin{align*}
\Gamma \left( -\frac{1}{2} \right) &= - 2 \sqrt{\pi}, \\
\frac{N-1}{2} \Gamma \left( \frac{N-1}{2} \right) &= \Gamma \left( \frac{N+1}{2} \right).
\end{align*}
When $N=2$, we conclude under the constraint
\begin{align*}
\mu &< \frac{1}{2} C\left( 2,\frac{1}{2} \right) C_{2,1/2,1/2} 
= \frac{1}{2} \frac{\Gamma \left( \frac{3}{4} \right)^2}{\Gamma \left( \frac{1}{4} \right)^2} \frac{2 \sqrt{\pi}}{\Gamma \left(\frac{3}{2} \right)} = 2 \frac{\Gamma \left( \frac{3}{4} \right)^2}{\Gamma \left( \frac{1}{4} \right)^2}.
\end{align*}

\end{proof}

Thanks to Lemma \ref{norm-equiv} we can introduce the norm $\| u \|_\mu := \sqrt{Q_\mu(u)}$ on $H^{1/2} (\R^N)$ for $0 < \mu < \mu^*$. In the rest of the paper we will use $\langle \cdot, \cdot \rangle$ for the scalar product corresponding to $Q$ and $\langle \cdot, \cdot \rangle_\mu$ for the scalar product which corresponds to $Q_\mu$. Moreover we define
\begin{align*}
	\cD(u) := \int_{\R^N \times \R^N} \frac{F(x,u(x))F(y,u(y))}{|x-y|^{N-\alpha}} \, dx \, dy,
\end{align*}
that is well-defined on $H^{1/2} (\R^N)$ by (F1) and (N). Thus we can rewrite our functional in the form
\begin{align*}
	\cE(u) = \frac12 \|u\|^2_\mu - \frac12 \cD(u) + \frac{1}{q} \int_{\R^N} K(x) |u(x)|^q \, dx.
\end{align*}
It is standard to check that $\cE$ is of $\cC^1$-class and its critical points are weak solutions to \eqref{eq:1.1}.

\section{Existence and boundedness of a Cerami sequence}

Suppose that $(E, \| \cdot \|)$ is a Hilbert space and $\cE \colon E \rightarrow \R$ is a nonlinear functional of the general form
\begin{align*}
	\cE(u) = \frac12 \|u\|^2 - \cI(u),
\end{align*}
where $\cI$ is of $\cC^1$ class and $\cI(0)=0$. We introduce the following set
\begin{align*}
	\cN := \{ u \in E \setminus \{ 0 \} \mid  \cE'(u)(u) = 0 \},
\end{align*}
which is known as the \textit{Nehari manifold}. It is obvious that any nontrivial critical point of $\cE$ belongs to $\cN$.
The following theorem follows from \cite{BM, MSS}, see also the abstract setting in \cite{B-prep}.

\begin{Th}[{\cite[Theorem 5.1]{B-prep}}]
Suppose that
\begin{itemize}
\item[(J1)] there is $r > 0$ such that
\begin{align*}
	\inf_{\|u\|=r} \cE(u) > 0;
\end{align*}
\item[(J2)] $\frac{\cI (t_n u_n)}{t_n^2} \to +\infty$ for $t_n \to +\infty$ and $u_n \to u \neq 0$;
\item[(J3)] for all $t > 0$ and $u \in \cN$ there holds
\begin{align*}
	\frac{t^2-1}{2} \cI'(u)(u) - \cI(tu) + \cI(u) \leq 0.
\end{align*}
\end{itemize}
Then $\cN \neq \emptyset$, $\Ga \neq \emptyset$ and
\begin{align*}
	c := \inf_{\cN} \cE = \inf_{\gamma \in \Gamma} \sup_{t \in [0,1]} \cE(\gamma(t)) = \inf_{u \in E \setminus \{0\}} \sup_{t \geq 0} \cE(tu) > 0,
\end{align*}
where
\begin{align*}
	\Gamma := \left\{ \gamma \in \cC ([0,1], E) \mid \gamma(0) = 0, \ \|\gamma(1)\| > r, \ \cE(\gamma(1)) < 0 \right\}.
\end{align*}
Moreover there is a Cerami sequence for $\cE$ at level~$c$, i.e. a sequence $\{ u_n \}_n \subset E$ such that
\begin{align*}
	\cE(u_n) \to c, \quad (1+\|u_n\|) \cE'(u_n) \to 0.
\end{align*}
\end{Th}

Observe that assumptions (J1)--(J3) do not require that $\cI(u) \geq 0$ for $u \in E$. In fact, $\cI$ may change its sign, which is possible in our situation. Moreover the theorem gives equivalent min-max-type characterizations of the level $c$ (which, in our situation, will be a critical value). 
\begin{Lem}\label{D_estimate}
There is $C > 0$ such that
\begin{align*}
	\cD(u) \leq C \left( \|u\|_\mu^4 + \|u\|_\mu^{p+2} + \|u\|_\mu^{2p} \right)
\end{align*}
for all $u \in H^{1/2}(\R^N)$.
\end{Lem}

\begin{proof}
In the proof, $C$ denotes a generic, positive constant which may vary from line to line. Fix any $\varepsilon > 0$. Using \eqref{F-eps} we have 
\begin{align*}
	\cD(u) &\leq \int_{\R^N \times \R^N} \frac{|F(x,u(x))||F(y,u(y))|}{|x-y|^{N-\alpha}} \, dx \, dy \\
	&\leq  \int_{\R^N \times \R^N} \frac{\left(\eps|u(x)|^2 + C_{\eps}|u(x)|^p\right)\left(\eps|u(y)|^2 + C_{\eps}|u(y)|^p\right)}{|x-y|^{N-\alpha}} \, dx \, dy.
\end{align*}
Applying the Hardy-Littlewood-Sobolev inequality we obtain
\[
	\cD(u) \leq C \left| \eps|u|^2 + C_{\eps}|u|^p \right|_r \left|\eps|u|^2 + C_{\eps}|u|^p \right|_r = C \left| \eps|u|^2 + C_{\eps}|u|^p \right|_r^2,
\]
where $r = \dfrac{2N}{N+\al}$.
By Minkowski inequality we obtain
\begin{align*}
	\cD(u) &\leq  C \left(\eps \left|u^2\right|_r + C_{\eps} \left|u^p\right|_r\right)^2 = C \left(\eps\left|u\right|^2_{2r} + C_{\eps}\left|u\right|^p_{pr}\right)^2 \\
	&= C \left( \eps^2 \left|u\right|^4_{2r} + 2 \eps C_{\eps}\left|u\right|^2_{2r}\left|u\right|^p_{pr} + C_{\eps}^2 \left|u\right|^{2p}_{pr}\right).
\end{align*}
From (N) there follows that $pr < \frac{2N}{N-1}$, and by Sobolev embeddings
\begin{align*}
	\cD(u) \leq C \left( \eps^2 \|u\|^4_{\mu} + 2 \eps C_{\eps}\|u\|^{p+2}_\mu + C_{\eps}^2\|u\|^{2p}_{\mu}\right),
\end{align*}
and therefore
\begin{align*}
	\cD(u) \leq C \left(\|u\|^4_{\mu} + \|u\|^{p+2}_\mu + \|u\|^{2p}_{\mu}\right).
\end{align*}
\end{proof}
Put
\begin{align*}
	\cI(u) :=   \frac12 \cD(u) - \frac{1}{q} \int_{\R^N} K(x) |u(x)|^q \, dx.
\end{align*}
\begin{Lem}\label{lem_assumpt}
Suppose (N), (F1)--(F4) and ($K$) are satisfied. Then (J1)--(J3) are satisfied on $E = H^{1/2}(\R^N)$ with the norm $\| \cdot \|_\mu$.
\end{Lem}

\begin{proof}
(J1) Since $\cI(u) =   \frac12 \cD(u) - \frac1q \int_{\R^N} K(x) |u(x)|^q \, dx$, by Lemma \ref{D_estimate} we have
\begin{align*}
	\cI(u) \leq   \frac12 \cD(u) \leq C \left( \|u\|_\mu^4 + \|u\|_\mu^{p+2} + \|u\|_\mu^{2p} \right) = C \|u\|_\mu^2 \left( \|u\|_\mu^2 + \|u\|_\mu^{p} + \|u\|_\mu^{2p-2} \right).
\end{align*}
Note that for $r > 0$ we have
\begin{align*}
	\cI(u) \leq C \|u\|_\mu^2 \left( r^2 + r^{p} + r^{2p-2} \right)
\end{align*}
for $\|u\|_\mu \leq r$. Put $A(r) := C \left( r^2 + r^{p} + r^{2p-2} \right)$ and then
\begin{align*}
	\cI(u) \leq A(r) \|u\|_\mu^2.
\end{align*}
Note that $A \colon [0, +\infty) \to [0, +\infty)$ is a continuous
function with $A(0) = 0$ and $\lim_{r \to +\infty} A(r) =
+\infty$. Hence we can take $r >0$ such that $A(r) = \frac14$, that is
\begin{align*}
	\cI(u) \leq \frac14 \|u\|_\mu^2. 
\end{align*}
Hence, for $\|u\|_\mu = r$
\begin{align*}
  \cE(u) = \frac12 \|u\|_\mu^2 - \cI(u) \geq \frac14 \|u\|_\mu^2 = \frac14 r^2 > 0.
\end{align*}
(J2) Since $t_n \to +\infty$, we may assume that $t_n \geq 1$ and we
recall that $q>2$. Then
\begin{multline*}
	\frac{\cI(t_nu_n)}{t_n^2} \geq \frac{\cI(t_nu_n)}{t_n^{q}} \\ 
	=  \frac12 \int_{\R^N \times \R^N} \frac{F(x,t_nu_n(x))F(y,t_nu_n(y))}{t_n^{q}|x-y|^{N-\alpha}} \, dx \, dy - \frac{1}{q} \int_{\R^N} K(x) |u_n(x)|^q \, dx \to +\infty
\end{multline*}
by (F3) and Fatou's lemma.

(J3) Let $u \in \cN$ and define
\begin{align*}
	\vp(t)=\dfrac{t^2-1}{2}\cI'(u)u - \cI(tu) + \cI(u)
\end{align*}
for $t \geq 0$; we remark that $\vp(1)=0$. Moreover, $\|u\|_\mu^2 = \cI'(u)u > 0$, which is equivalent to
\begin{equation} \label{in_1}
	 \int_{\R^N \times \R^N} \frac{F(x,u(x))f(y,u(y))u(y)}{|x-y|^{N-\alpha}} \, dx \, dy > \int_{\R^N} K(x) |u(x)|^q \, dx.
\end{equation}
We compute
\begin{align*}
	&\quad \dfrac{d\vp(t)}{dt}=t\cI'(u)u - \cI'(tu)u \\
	&= t \int_{\R^N \times \R^N} \frac{F(x,u(x))f(y,u(y))u(y)}{|x-y|^{N-\alpha}} \, dx \, dy - t\int_{\R^N} K(x) |u(x)|^q \, dx \\
	&\quad-  \int_{\R^N \times \R^N} \frac{F(x,tu(x))f(y,tu(y))u(y)}{|x-y|^{N-\alpha}} \, dx \, dy  + \int_{\R^N} K(x) |tu(x)|^q \, dx\\
	&= \int_{\R^N \times \R^N} \left[\frac{F(x,u(x))f(y,u(y))tu(y)}{|x-y|^{N-\alpha}} - \frac{F(x,tu(x))f(y,tu(y))u(y)}{|x-y|^{N-\alpha}}\right] \, dx \, dy\\
	&\quad+\left(t^{q-1}-t\right)\int_{\R^N} K(x) |u(x)|^q \, dx.
\end{align*}
For almost every fixed $x,y \in \R^N$ we define the map $\psi\colon (0,+\infty) \to \R$ as
\begin{equation}\label{psi}
	\psi(t) := \psi_{(x,y)}(t) := \frac{F(x,tu(x))f(y,tu(y))u(y)}{t^{q-1}}.
\end{equation}
If $t<1$ then by \eqref{in_1}, and using \eqref{psi}, we have
\begin{align*}
	&\dfrac{d\vp(t)}{dt}= \cI'(u)(tu)-\cI'(tu)(u) \\
	&\geq  \int_{\R^N \times \R^N} \left[\frac{F(x,u(x))f(y,u(y))tu(y)}{|x-y|^{N-\alpha}} - \frac{F(x,tu(x))f(y,tu(y))u(y)}{|x-y|^{N-\alpha}}\right. \\
	&\quad\left.+ \frac{t^{q-1}F(x,u(x))f(y,u(y))u(y)}{|x-y|^{N-\alpha}} - \frac{F(x,u(x))f(y,u(y))tu(y)}{|x-y|^{N-\alpha}}\right]	\, dx \, dy\\
	&= \int_{\R^N \times \R^N} \left[\frac{t^{q-1}F(x,u(x))f(y,u(y))u(y)}{|x-y|^{N-\alpha}} - \frac{F(x,tu(x))f(y,tu(y))u(y)}{|x-y|^{N-\alpha}}\right] \, dx \, dy\\
	&= t^{q-1} \int_{\R^N \times \R^N} \frac{\psi_{(x,y)}(1) - \psi_{(x,y)}(t)}{|x-y|^{N-\alpha}} \, dx \, dy.
\end{align*}
We claim that $\psi(t) \geq 0$ and $\psi$ is non-decreasing on $(0, 1]$, which completes the proof in the case $t < 1$. From (F3) and \eqref{AR-q} there follows that $\psi(t) \geq 0$. We rewrite \eqref{psi} in the form
\begin{gather*}
\psi(t) = \frac{F(x,tu(x))}{t^{\frac{q}{2}}} \frac{f(y,tu(y))u(y)}{t^{\frac{q}{2}-1}}
\end{gather*}
Again, taking \eqref{AR-q} into account
\begin{align*}
	\frac{d}{dt} \left[ \frac{F(x,tu(x))}{t^{\frac{q}{2}}} \right] = \frac{f(x,tu(x))tu(x) - \frac{q}{2}F(x,tu(x))}{t^{\frac{q}{2}+1}} \geq 0,
\end{align*}
so $t \mapsto \frac{F(x,tu(x))}{t^{\frac{q}{2}}}$ is non-decreasing on $(0,+\infty)$. Moreover
\begin{align*}
	\frac{f(y,tu(y))u(y)}{t^{\frac{q}{2}-1}} = \frac{f(y,tu(y))u(y)}{|t u(y)|^{\frac{q}{2}-1}} |u(y)|^{\frac{q}{2}-1} = \frac{f(y,tu(y))}{|t u(y)|^{\frac{q}{2}-1}} |u(y)|^{\frac{q}{2}-1} u(y)
\end{align*}
so that 
\begin{align*}
	t \mapsto \frac{f(y,tu(y))u(y)}{t^{\frac{q}{2}-1}}
\end{align*}
is non-decreasing by (F4), if $u(y) \neq 0$. Hence $\psi$ is non-decreasing as a product of non-negative, non-decreasing functions. Hence
\begin{align*}
	\dfrac{d\vp(t)}{dt}=\cI'(u)(tu) - \cI'(tu)(u) \geq  t^{q-1} \int_{\R^N \times \R^N} \frac{\psi_{(x,y)}(1)-\psi_{(x,y)}(t)}{|x-y|^{N-\alpha}} \, dx \,dy \geq 0,
\end{align*}
that is $\vp(t) \leq \vp(1)=0$ for $t \in (0,1]$.

Similarly we show that if $t \in (1,+\infty)$ then
\begin{align*}
	\dfrac{d\vp(t)}{dt}=\cI'(u)(tu) - \cI'(tu)(u) \leq 0,
\end{align*}
therefore $\vp(t) \leq \vp(1)=0$ for $t \in (1,+\infty)$.
\end{proof}

\begin{Lem}
\label{Cer_bdd}
Any Cerami sequence $\{ u_n\}_n$ for $\cE$ is bounded.
\end{Lem}

\begin{proof}
	By the properties of Cerami's sequences, we may write
	\begin{align*}
	\limsup_{n \to +\infty} \cE(u_n) &= \limsup_{n \to +\infty} \left( \cE(u_n) - \frac{1}{q} \cE'(u_n) u_n \right) \\
	&= \limsup_{n \to +\infty} \left[\left(\frac12-\frac1q\right)\|u_n\|^2_{\mu} \right. \\
	& \quad \left. {} + \frac1q \int_{\R^N \times \R^N} \frac{F(x,u_n(x)) \left[  f(y,u_n(y))u_n(y) - \frac{q}{2} F(y,u_n(y)) \right]  }{|x-y|^{N-\alpha}} \, dx \, dy \right] \\
	&\geq \limsup_{n \to +\infty} \left(\frac12-\frac1q\right)\|u_n\|^2_{\mu}
	\end{align*}
	by \eqref{AR-q}. Since $\limsup_{n \to +\infty} \cE(u_n)$ is finite, the proof is complete.
	
\end{proof}

\section{Decomposition of bounded minimizing sequences}

To ease notation we set
\begin{equation*}
  I_\alpha (x) := \frac{1}{|x|^{N-\alpha}}, \quad x \in \mathbb{R}^N \setminus \{0\}.
\end{equation*}

\begin{Lem}\label{Brezis-Lieb-D}
  Suppose that $\{u_n\}_n \subset H^{1/2}(\R^N)$ is a bounded sequence
  such that $u_n \weakto u_0$ in $H^{1/2} (\R^N)$. Then
\begin{equation*}
\cD (u_n - u_0) - \cD(u_n) + \cD(u_0)   \to 0 \ \mbox{as} \ n \to +\infty.
\end{equation*}
\end{Lem}

The proof is similar to the proof of \cite[Lemma 2.2]{CassaniZhang} in
the space $H^1 (\R^N)$, so we omit it.

\begin{Lem}\label{BrezisLiebDecomp}
  Suppose that $\{u_n\}_n \subset H^{1/2}(\R^N)$ and there are
  $\ell \geq 0$, $\{z_n^k\}_n \subset \Z^N$ and
  $w^k \in H^{1/2} (\R^N)$ for $k=1,\ldots, \ell$ such that
  $u_n \weakto u_0$, $u_n(\cdot - z_n^k) \weakto w^k$ in
  $x \in H^{1/2}(\R^N)$ and
\begin{equation}\label{x}
\left\| u_n - u_0 - \sum_{k=1}^\ell w^k(\cdot - z_n^k) \right\| \to 0.
\end{equation}
Then
\begin{equation*}
\cD(u_n) \to \cD(u_0) + \sum_{k=1}^\ell \cD(w^k).
\end{equation*}
\end{Lem}

\begin{proof}
For any $m \in \{ 0, 1, \ldots, \ell\}$ we introduce
\begin{equation*}
a_m^n := u_n - u_0 - \sum_{k=1}^m w^k \left( \cdot - z_n^k \right).
\end{equation*}
From Lemma \ref{Brezis-Lieb-D} we have that
\begin{equation*}
\cD(a_0^n) - \cD(u_n) + \cD(u_0) \to 0.
\end{equation*}
Taking $a_0^n \left( \cdot + z_n^1 \right)$ as $u_n$ and $w^1$ as
$u_0$ in Lemma \ref{Brezis-Lieb-D} we obtain
\begin{equation*}
\cD \left( a_0^n \left( \cdot + z_n^1 \right) - w^1 \right) - \cD \left( a_0^n \left( \cdot + z_n^1 \right) \right) + \cD(w^1) \to 0
\end{equation*}
or equivalently
\begin{equation}\label{D1}
\cD (a_1^n) - \cD(a_0^n) + \cD(w^1) \to 0.
\end{equation}
Similarly
\begin{equation}\label{D2}
\cD(a_2^n) - \cD(a_1^n) + \cD(w^2) \to 0.
\end{equation}
Combining \eqref{D1} and \eqref{D2} gives
\begin{gather*}
\cD(a_2^n) + \cD(w^2) + \cD(w^1) - \cD(a_0^n) \to 0.
\end{gather*}
Iterating the same reasoning we obtain
\begin{gather*}
\cD(a_\ell^n) + \sum_{k=1}^\ell \cD(w^k) - \cD(a_0^n) \to 0.
\end{gather*}
Taking into account that $a_0^n = u_n - u_0$ we see that
\begin{equation}\label{D3}
\cD(a_\ell^n) + \sum_{k=1}^\ell \cD(w^k) - \cD(u_n-u_0) \to 0.
\end{equation}
In view of \eqref{x} we obtain $a_\ell^n \to 0$, so that
\begin{gather*}
\cD(a_\ell^n) \to 0
\end{gather*}
and \eqref{D3} reads as 
\begin{gather*}
\cD(u_n-u_0) \to \sum_{k=1}^\ell \cD(w^k).
\end{gather*}
Taking again Lemma \ref{Brezis-Lieb-D} into account we obtain
\begin{gather*}
\cD(u_n) \to \cD(u_0) + \sum_{k=1}^\ell \cD(w^k).
\end{gather*}
\end{proof}

\begin{Lem}\label{weak-conv-D}
  $\cD' \colon H^{1/2} (\R^N) \rightarrow \left( H^{1/2}(\R^N) \right)^* =
  H^{-1/2}(\R^N)$ is weak-to-weak* continuous, i.e. if $\{u_n\}_n$ is
  bounded and $u_n \weakto u_0$ in $H^{1/2} (\R^N)$ and
  $\varphi \in H^{1/2} (\R^N)$ then
\begin{gather*}
\cD'(u_n)(\varphi) \to \cD'(u_0)(\varphi).
\end{gather*}
\end{Lem}

\begin{proof}
Observe that
\begin{gather*}
\cD'(u_n)(\varphi) = \int_{\R^N} \left( I_\alpha * (F(\cdot, u_n (\cdot)) \right)(x) f(x,u_n(x))\varphi(x) \, dx,
\end{gather*}
which (for brevity) we will write shortly as
\begin{gather*}
\cD'(u_n)(\varphi) = \int_{\R^N} (I_\alpha * (F \circ u_n)) (f \circ u_n) \varphi \, dx.
\end{gather*}
Since $\{u_n\}_n$ is bounded in $H^{1/2} (\R^N)$, it is bounded also
in $L^2 (\R^N) \cap L^\frac{2N}{N-1} (\R^N)$. From \eqref{F-eps} there
follows that $\{ F (\cdot, u_n (\cdot) ) \}_n$ is bounded in
$L^{\frac{2N}{N+\alpha}} (\R^N)$, since (N) implies that
$p \cdot \frac{2N}{N+\alpha} < \frac{2N}{N-1}$. The weak convergence
$u_n \weakto u_0$ in $H^{1/2}(\R^N)$ implies that $u_n(x) \to u_0(x)$
for a.e. $x \in \R^N$. Hence $F(x,u_n(x)) \to F(x,u_0(x))$ for a.e.
$x \in \R^N$, and therefore
$F(\cdot, u_n (\cdot)) \weakto F(\cdot, u_0 (\cdot))$ in
$L^{\frac{2N}{N+\alpha}} (\R^N)$. From the Hardy-Littlewood-Sobolev
inequality we obtain that
\begin{gather*}
I_\alpha * (F \circ u_n) \weakto I_\alpha * (F \circ u_0) \mbox{ in } L^\frac{2N}{N-\alpha} (\R^N).
\end{gather*}
Moreover, from (F1), $f \circ u_n \to f \circ u_0$ in
$L^{\frac{2N}{N+\alpha}}_{\loc} (\R^N)$. Hence, for any
$\varphi \in \cC_0^\infty (\R^N)$ there holds
\begin{gather*}
  (I_\alpha * (F \circ u_n)) (f \circ u_n) \varphi \to (I_\alpha * (F
  \circ u_0)) (f \circ u_0) \varphi \mbox{ in } L^1 (\R^N),
\end{gather*}
i.e. $\cD'(u_n)(\varphi) \to \cD'(u_0)(\varphi)$.
\end{proof}

\begin{Cor}\label{weak-conv-E}
  $\cE' \colon H^{1/2} (\R^N) \rightarrow \left( H^{1/2}(\R^N)
  \right)^* = H^{-1/2}(\R^N)$ is weak-to-weak* continuous.
\end{Cor}
\begin{proof}
Indeed, take $u_n \weakto u_0$ in $H^{1/2} (\R^N)$ and $\varphi \in \cC_0^\infty (\R^N)$, we have that 
\begin{gather*}
\cE'(u_n)(\varphi) = \langle u_n, \varphi \rangle_\mu - \frac12 \cD'(u_n)(\varphi) + \int_{\R^N} K (x) |u_n(x)|^{q-2} u_n(x) \varphi(x) \, dx.
\end{gather*}
We may assume that, up to a subsequence, $u_n(x) \to u_0(x)$ for a.e. $x \in \R^N$. Obviously $\langle u_n, \varphi \rangle \to \langle u_0, \varphi \rangle_\mu$. Moreover, from Lemma \ref{weak-conv-D}, $\cD'(u_n)(\varphi) \to \cD'(u_0)(\varphi)$. Note that for any measurable subset $E \subset \supp \varphi$,
\begin{align*}
\left| \int_{E} K (x) |u_n(x)|^{q-2} u_n(x) \varphi(x) \, dx \right| \leq | K  |_\infty |u_n|^{q-1}_q | \varphi \chi_E |_q,
\end{align*}
so that in view of the Vitali convergence theorem
\begin{gather*}
\int_{\R^N} K (x) |u_n(x)|^{q-2} u_n(x) \varphi(x) \, dx \to \int_{\R^N} K (x) |u_0(x)|^{q-2} u_0(x) \varphi(x) \, dx
\end{gather*}
and we conclude.
\end{proof}

Define
\begin{equation}\label{Eper}
\cE_{\mathrm{per}} (u) := \cE(u) - \frac{1}{2} \int_{\R^N} V_{l}(x) u^2 \, dx + \frac{\mu}{2} \int_{\R^N} \frac{u^2}{|x|} \, dx.
\end{equation}
Note that $\cE_{\mathrm{per}} (u(\cdot - z)) = \cE_{\mathrm{per}}(u)$ for any $z \in \Z^N$.

\begin{Th}\label{ThSplitting}
Let $\{u_n\}_n$ be a bounded Palais-Smale sequence. Then (up to a subsequence) there is an integer $\ell \geq 0$ and sequences $(z_n^k) \subset \mathbb{Z}^{N}$, $w^k \in H^{1/2} (\R^{N})$, $k =1,\ldots, \ell$ such that
\begin{enumerate}
\item[(i)] $u_n \weakto u_0$ and $\cE'(u_0) = 0$;
\item[(ii)] $|z_n^k| \to +\infty$ and $|z_n^k - z_n^{k'}| \to +\infty$ for $k \neq k'$;
\item[(iii)] $w^k \neq 0$ and $\cE_{\mathrm{per}} ' (w^k) = 0$ for $1 \leq k \leq \ell$;
\item[(iv)] $u_n - u_0 - \sum_{k=1}^\ell w^k (\cdot - z_n^k) \to 0$;
\item[(v)] $\cE(u_n) \to \cE(u_0) + \sum_{k=1}^\ell \cE_{\mathrm{per}} (w^k)$.
\end{enumerate}
\end{Th}

\begin{proof}
\textbf{Step 1: } \textit{Up to a subsequence, $u_n \weakto u_0$ and $\cE'(u_0) = 0$.} \\
Since $\{u_n\}_n$ is bounded, $u_n \weakto u_0$ up to a subsequence. Taking into account that $\cE'(u_0) \to 0$ and Corollary \ref{weak-conv-E} we obtain $\cE'(u_0) = 0$. \\
\textbf{Step 2: } \textit{Let $v_n^1 := u_n - u_0$. Suppose that}
\[
\lim_{n \to +\infty} \sup_{z \in \R^N} \int_{B(z,1)} |v_n^1(x)|^2 \, dx = 0.
\]
\textit{Then $u_n \to u_0$ and the statement follows with $\ell = 0$.} \\
Note that
\begin{equation*}
\cE'(u_n)(v_n^1) = \|v_n^1\|^2_\mu + \langle u_0, u_n-u_0 \rangle_\mu  - \frac12 \cD'(u_n)(v_n^1) + \int_{\R^N}K(x) |u_n|^{q-2} u_n v_n^1 \, dx.
\end{equation*}
Taking into account that
\[
\langle u_0, u_n - u_0 \rangle_\mu -  \frac12 \cD'(u_0)(v_n^1) + \int_{\R^N}K(x) |u_0|^{q-2} u_0 v_n^1 \, dx = 0
\]
we have
\begin{align*}
\|v_n^1\|^2_\mu &= \cE'(u_n)(v_n^1) - \langle u_0, u_n-u_0 \rangle_\mu + \frac12 \cD'(u_n)(v_n^1) - \int_{\R^N} K(x) |u_n|^{q-2} u_n v_n^1 \, dx \\
&= \cE'(u_n)(v_n^1) - \frac12 \cD'(u_0)(v_n^1) + \int_{\R^N}K(x) |u_0|^{q-2} u_0 v_n^1 \, dx  \\
&\quad +  \frac12 \cD'(u_n)(v_n^1) - \int_{\R^N} K(x) |u_n|^{q-2} u_n v_n^1 \, dx.
\end{align*}
From the boundedness of $\{ v_n^1\}_n$ there follows that
\[
\cE'(u_n)(v_n^1) \to 0.
\]
From H\"older's inequality and Lions' Concentration-Compactness Principle \cite{L} we obtain
\begin{align*}
\left| \int_{\R^N}K(x) |u_0|^{q-2} u_0 v_n^1 \, dx \right| \leq |K|_\infty |u_0|^{q-1}_q |v_n^1|_q \to 0, \\
\left| \int_{\R^N}K(x) |u_n|^{q-2} u_n v_n^1 \, dx \right| \leq |K|_\infty |u_n|^{q-1}_q |v_n^1|_q \to 0.
\end{align*}
Moreover, from the Hardy-Littlewood-Sobolev inequality there follows that
\[
\cD'(u_n)(v_n^1) = \cD'(u_n)(u_n - u_0) \to 0 \quad \mbox{and} \quad \cD'(u_0)(v_n^1) = \cD'(u_0)(u_n - u_0) \to 0.
\]
Thus $\|v_n^1\|^2_\mu \to 0$ and $u_n \to u_0$. Hence $\cE(u_n) \to \cE(u_0)$ and the statement of the theorem holds true for $\ell = 0$. \\
\textbf{Step 3: } \textit{Suppose that there is a sequence $\{z_n\}_n \subset \Z^N$ such that}
\[
\liminf_{n\to+\infty} \int_{B(z_n, 1+\sqrt{N})} |v_n^1|^2 \, dx > 0.
\]
\textit{Then there is $w \in H^{1/2} (\R^N)$ such that (up to a subsequence):
\begin{itemize}
\item[(i)] $|z_n| \to +\infty$; 
\item[(ii)] $u_n (\cdot + z_n) \weakto w \neq 0$;
\item[(iii)] $\cE_{\per}' (w) = 0$.
\end{itemize}
}
Statements (i) and (ii) are standard, so we will show only (iii). Put $v_n := u_n (\cdot + z_n)$. Similarly as in Step 1 we see that
\[
\cE_{\per}' (v_n) (\varphi) \to \cE_{\per}' (w) (\varphi)
\]
for any $\varphi \in \cC_0^\infty (\R^N)$. Moreover
\begin{align*}
o(1) &= \cE'(u_n)(\varphi (\cdot - z_n)) = \cE_{\per}'(v_n)(\varphi) + \int_{\R^N} V_l (x+z_n) v_n \varphi \, dx - \mu \int_{\R^N} \frac{u_n \varphi(\cdot - z_n)}{|x|} \, dx \\
&= \cE_{\per}' (w) (\varphi) + \int_{\supp \varphi} V_l (x+z_n) v_n \varphi \, dx - \mu \int_{\R^N} \frac{u_n \varphi(\cdot - z_n)}{|x|} \, dx + o(1)
\end{align*}
and
\[
\cE_{\per}' (w) (\varphi) = - \int_{\supp \varphi} V_l (x+z_n) v_n \varphi \, dx + \mu \int_{\R^N} \frac{u_n \varphi(\cdot - z_n)}{|x|} \, dx + o(1).
\]
Lemma \ref{Hardy} implies that $\{ u_n \}_n$ is bounded in $L^2 (\R^N; |x|^{-1} dx)$, and from \cite[Lemma 2.5]{B} we obtain that
\[
\left| \int_{\R^N} \frac{u_n \varphi(\cdot - z_n)}{|x|} \, dx \right| \leq \left( \int_{\R^N} \frac{|u_n|^2}{|x|} \, dx \right)^{1/2} \left( \int_{\R^N} \frac{|\varphi(x - z_n)|^2}{|x|} \, dx \right)^{1/2} \to 0.
\]
Hence it is sufficient to show that $\int_{\supp \varphi} V_l (x+z_n) v_n \varphi \, dx \to 0$. Fix any measurable set $E \subset \supp \varphi$. From the H\"older inequality we obtain that
\[
\int_E |V_l (x+z_n) v_n \varphi| \, dx \leq |V_l|_\infty | v_n|_2 |\varphi \chi_E|_2.
\]
Then, the boundedness of $\{v_n\}_n$ in $L^2 (\R^N)$ implies that the family $\{ V_l (\cdot +z_n) v_n \varphi \}_n$ is uniformly integrable on $\supp \varphi$ and from Vitali's convergence theorem we derive
\[
\int_{\supp \varphi} V_l (x+z_n) v_n \varphi \, dx \to 0,
\]
and the proof of Step 3 is completed. \\
\textbf{Step 4:} \emph{Suppose that there are $m \geq 1$, $\{ z_n^k \}_n \subset \Z^N$, $w^k \in H^{1/2}(\R^N)$ for $k \in \{1, 2, \ldots, m \}$ such that}
\begin{align*}
|z_n^k| \to +\infty, \ |z_n^k - z_n^{k'}| \to +\infty &\mbox{ for } 1 \leq k < k' \leq m; \\
u_n(\cdot + z_n^k) \to w^k \neq 0 &\mbox{ for } 1 \leq k \leq m; \\
\cE_{\per}' (w^k) = 0 &\mbox{ for } 1 \leq k \leq m.
\end{align*}
\emph{Then}
\begin{itemize}
\item[(1)] if 
\begin{equation}\label{it-1}
\sup_{z \in \R^N} \int_{B(z, 1)} \left| u_n - u_0 - \sum_{k=1}^m w^k(\cdot - z_n^k) \right|^2 \, dx \to 0 \mbox{ as } n \to +\infty
\end{equation}
then
\[
u_n - u_0 - \sum_{k=1}^m w^k(\cdot - z_n^k)  \to 0;
\]
\item[(2)] if there is $\{z_n^{m+1} \}_n \subset \Z^N$ such that
\begin{equation}\label{it-2}
\liminf_{n\to+\infty} \int_{B(z_n^{m+1}, 1+\sqrt{N})} \left| u_n - u_0 - \sum_{k=1}^m w^k(\cdot - z_n^k) \right|^2 \, dx > 0
\end{equation}
then there is $w^{m+1} \in H^{1/2} (\R^N)$ such that (up to subsequences)
\begin{itemize}
\item[(i)] $|z_n^{m+1}| \to +\infty$, $|z_n^{m+1} - z_n^k| \to +\infty$ for $1 \leq k \leq m$,
\item[(ii)] $u_n (\cdot - z_n^{m+1}) \weakto w^{m+1} \neq 0$,
\item[(iii)] $\cE_{\per}'(w^{m+1}) = 0$.
\end{itemize}
\end{itemize}
Define
\begin{equation*}
\xi_n := u_n - u_0 - \sum_{k=1}^m w^k(\cdot - z_n^k).
\end{equation*}
Suppose that \eqref{it-1} holds. Then Lions' Concentration-Compactness
Principle implies that $\xi_n \to 0$ in $L^t (\R^N)$ for any
$2 < t < \frac{2N}{N-1}$. Repeating the argument of \cite[Lemma 4.3,
Step 4]{BS} and using the identity $\cE'(u_0)(\xi_n) = 0$, we obtain
that
\begin{align*}
  \|\xi_n\|^2_\mu &= -\langle u_0, \xi_n \rangle_\mu - \sum_{k=1}^m \langle w^k (\cdot - z_n^k), \xi_n \rangle_\mu \\
                  &\quad +\frac12 \cD'(u_n)(\xi_n) - \int_{\R^N} K(x) |u_n|^{q-2} u_n \xi_n \, dx + o(1) \\
                  &= - \frac12 \cD'(u_0)(\xi_n) + \int_{\R^N} K(x) |u_0|^{q-2} u_0 \xi_n \, dx - \sum_{k=1}^m \langle w^k (\cdot - z_n^k), \xi_n \rangle_\mu \\ &\quad + \frac12\cD'(u_n)(\xi_n) - \int_{\R^N} K(x) |u_n|^{q-2} u_n \xi_n \, dx + o(1).
\end{align*}
Taking into account that each $w^k$ is a critical point of
$\cE_{\per}$, we obtain that
\begin{align*}
  \|\xi_n\|^2_\mu &= \frac12 \cD'(u_n)(\xi_n) - \frac12 \cD'(u_0)(\xi_n) - \frac12 \sum_{k=1}^m \cD'(w^k (\cdot - z_n^k))(\xi_n) \\
                  &\quad - \int_{\R^N} K(x) \left( |u_n|^{q-2} u_n - |u_0|^{q-2} u_0 - \sum_{k=1}^m |w^k (x - z_n^k)|^{q-2} w^k (x - z_n^k) \right) \xi_n \, dx \\
                  &\quad - \sum_{k=1}^m \int_{\R^N} V_l (x) w^k (x - z_n^k) \xi_n \, dx + \mu \sum_{k=1}^m \int_{\R^N} \frac{w^k(x - z_n^k) \xi_n}{|x|} \, dx + o(1).
\end{align*}
From \cite[Lemma 2.5]{B}, we see that
\[
  \int_{\R^N} \frac{w^k(x - z_n^k) \xi_n}{|x|} \, dx \to 0.
\]
Similarly as in \cite[Lemma 4.3, Step 4]{BS} we get that
\[
  \int_{\R^N} K(x) \left( |u_n|^{q-2} u_n - |u_0|^{q-2} u_0 -
    \sum_{k=1}^m |w^k (x - z_n^k)|^{q-2} w^k (x - z_n^k) \right) \xi_n
  \, dx \to 0
\]
and 
\[
  \int_{\R^N} V_l (x) w^k (x - z_n^k) \xi_n \, dx \to 0.
\]
Hence
\[
  \|\xi_n\|^2 = \frac12 \cD'(u_n)(\xi_n) - \frac12 \cD'(u_0)(\xi_n) - \frac12 \sum_{k=1}^m
  \cD'(w^k (\cdot - z_n^k))(\xi_n) + o(1).
\]
To show that $|\cD'(u_n)(\xi_n)| \to 0$ we note that $(F1)$,
\eqref{F-eps}, Hardy-Littlewood-Sobolev and Minkowski inequalities
imply that
\begin{align}
\label{D'_est}
  |\cD'(u_n)(\xi_n)| &\leq 2\int_{\R^N \times \R^N} \dfrac{|F(x,u_n(x))||f(y,u_n(y))||\xi_n(y)|}{|x-y|^{N-\al}} \, dx \, dy \notag\\
                     &\leq C\int_{\R^N \times \R^N} \dfrac{\left(\eps|u_n|^2 + C_{\eps}|u_n|^p\right)\left(|u_n|^{\frac{\al}{N}} + |u_n|^{p-1} \right)|\xi_n|}{|x-y|^{N-\al}} \, dx \, dy \notag \\
                     &\leq C \left|\eps|u_n|^2 + C_{\eps}|u_n|^p\right|_r\left|\left(|u_n|^{\frac{\al}{N}} + |u_n|^{p-1} \right)|\xi_n|\right|_r \notag\\
                     &\leq C \left(|u_n|^2_{2r} + |u_n|^p_{pr} \right)\left(|u_n|^{\al/N}_{r \alpha / N} + |u_n|^{p-1}_{(p-1)r}\right)|\xi_n|_r \to 0,
\end{align}
where $r=\frac{2N}{N+\al} < 2N/(N-1)$. Similarly
$\cD'(u_0)(\xi_n) \to 0$ and $\cD'(w^k (\cdot - z_n^k))(\xi_n) \to
0$. Hence $\xi_n \to 0$ and the proof of Step 4 in this case is
completed. Hence, assume that for some $\{z_n^{m+1} \}_n \subset \Z^N$
we have \eqref{it-2}. (i) and (ii) are standard. Put
$v_n := u_n(\cdot + z_n^{m+1})$. As in Step 3 we see that
\begin{equation*}
\cE_{\per}' (v_n)(\varphi) - \cE_{\per}' (w^{m+1})(\varphi) \to 0
\end{equation*}
and $\cE_{\per}' (v_n)(\varphi) \to 0$ for any $\varphi \in \cC_0^\infty (\R^N)$. \\
\textbf{Step 5:} \textit{Conclusion.} Iterating Step 4 we construct
functions $w^k \neq 0$ and sequences $\{z_n^k\}_n$. Since $w^m$ are
critical points of $\cE_{\per}$, there is $\rho > 0$ such that
$\|w^k\| \geq \rho$.  Properties of the weak convergence yield
\begin{align*}
0 &\leq \lim_{n \to +\infty} \left\| u_n - u_0 - \sum_{k=1}^\ell w^k (\cdot - z_n^k) \right\|^2 \\
&= \lim_{n\to +\infty} \left( \| u_n \|^2 - \|u_0\|^2 - \sum_{k=1}^m \|w^k\|^2 \right) \leq \limsup_{n \to +\infty} \|u_n\|^2 - \|u_0\|^2 - m \rho^2.
\end{align*}
Hence
\begin{equation*}
  \rho^2 m \leq \limsup_{n \to+\infty} \|u_n\|^2 - \|u_0\|^2
\end{equation*}
and the procedure will finish after finite number --- say \(\ell\) --- of steps. \\
\textbf{Step 6:} \textit{(v) holds.} (v) follows from properties of the weak convergence, Hardy inequality and Lemma \ref{BrezisLiebDecomp}. Indeed,
\begin{align*}
	\cE(u_n) &= \dfrac{1}{2}\langle u_n,u_n\rangle_{\mu} -  \frac12 \cD(u_n) + \dfrac{1}{q}\int_{\R^N}K(x)|u_n(x)|^q \, dx \\
	& = \dfrac{1}{2}\langle u_0,u_0\rangle_{\mu} + \dfrac{1}{2}\langle u_n - u_0,u_n-u_0\rangle_{\mu} + \langle u_0,u_n-u_0\rangle_{\mu} - \frac12 \cD(u_n) \\
	& \quad + \dfrac{1}{q}\int_{\R^N}K(x)|u_n(x)|^q \, dx \\	
	& = \cE(u_0) + \cE_{\per}(u_n-u_0) + \langle u_0,u_n-u_0\rangle_{\mu} + \frac12 \cD(u_n-u_0) - \frac12 \cD(u_n) + \frac12 \cD(u_0) \\
	& \quad - \dfrac{1}{q}\int_{\R^N} K(x)\left[|u_n-u_0|^q + |u_0|^q - |u_n|^q \right] \, dx + \dfrac{1}{2}\int_{\R^N} V_l(x)(u_n-u_0)^2 \, dx \\
	& \quad - \dfrac{\mu}{2}\int_{\R^N} \dfrac{(u_n-u_0)^2}{|x|} \, dx.
\end{align*}
By the weak convergence we obtain 
\begin{equation*}
\langle u_0,u_n-u_0\rangle_{\mu} \to 0.
\end{equation*}
Lemma \ref{Brezis-Lieb-D} imply that 
\begin{equation*}
\cD(u_n-u_0) - \cD(u_n) + \cD(u_0) \to 0.
\end{equation*}
By a classical Brezis-Lieb lemma argument (see \cite[Proposition 4.7.30]{BO}) we have also that
\begin{equation*}
\int_{\R^N} K(x)\left[|u_n-u_0|^q + |u_0|^q - |u_n|^q \right] \, dx \to 0.
\end{equation*}
Let $E \subset \R^N$ a measurable set, by (V1) and H\"older inequality we have
\[
	\int_E |V_l(x)||u_n-u_0|^2 dx \leq |V_l\chi_E|_N|u_n-u_0|^2_{\frac{2N}{N-1}}
\]
and, since $(u_n-u_0)_n$ is bounded in $H^{1/2}(\R^N)$, by Vitali convergence theorem 
\begin{equation*}
\int_{\R^N} V_l(x)(u_n-u_0)^2 \, dx \to 0.
\end{equation*}
We note that
\begin{equation*}
 \int_{\R^N} \dfrac{(u_n-u_0)^2}{|x|} \, dx 
  =  \int_{\R^N} \dfrac{(u_n-u_0) u _n}{|x|} \, dx -  \int_{\R^N} \dfrac{(u_n-u_0) u_0}{|x|} \, dx 
\end{equation*}
and, as before, $\int_{\R^N} \frac{(u_n-u_0) u_0}{|x|} \, dx \to 0$. Moreover
\begin{align*}
\int_{\R^N} \dfrac{(u_n-u_0) u _n}{|x|} \, dx = \int_{\R^N} \dfrac{\left( u_n-u_0 - \sum_{k=1}^\ell w^k (\cdot - z_n^k) \right) u _n}{|x|} + \sum_{k=1}^\ell  \dfrac{  w^k (\cdot - z_n^k) u _n}{|x|} \, dx.
\end{align*}
Then
\begin{align*}
\left| \int_{\R^N} \dfrac{  w^k (\cdot - z_n^k) u _n}{|x|} \, dx\right|  \leq \left( \int_{\R^N} \dfrac{  |w^k (\cdot - z_n^k)|^2 }{|x|} \, dx \right)^{1/2} \left( \int_{\R^N} \dfrac{  |u_n|^2 }{|x|} \, dx \right)^{1/2} \to 0
\end{align*}
and
\begin{multline*}
\left| \int_{\R^N} \dfrac{\left( u_n-u_0 - \sum_{k=1}^\ell w^k (\cdot - z_n^k) \right) u _n}{|x|} \, dx \right| \\ \leq 
\left( \int_{\R^N} \dfrac{\left( u_n-u_0 - \sum_{k=1}^\ell w^k (\cdot - z_n^k) \right)^2}{|x|} \, dx \right)^2  \left( \int_{\R^N} \dfrac{  |u_n|^2 }{|x|} \, dx \right)^{1/2} \\ \leq  C \left\| u_n-u_0 - \sum_{k=1}^\ell w^k (\cdot - z_n^k)  \right\| \left( \int_{\R^N} \dfrac{  |u_n|^2 }{|x|} \, dx \right)^{1/2} \to 0.
\end{multline*}
Thus
\begin{equation*}
\cE(u_n) = \cE(u_0) + \cE_{\per}(u_n-u_0) + o(1)
\end{equation*}
and it is sufficient to show that $\cE_{\per}(u_n-u_0) \to \sum_{k=1}^\ell \cE_{\per}(w^k)$.

Now, we compute
\begin{align*}
	\cE_{\per}(u_n-u_0) &= \dfrac{1}{2}\|u_n-u_0\|^2_{\mu} - \frac12 \cD(u_n-u_0) + \dfrac{1}{q}\int_{\R^N} K(x)|u_n-u_0|^q \, dx \\
	& \quad - \dfrac{1}{2}\int_{\R^N} V_l(x)(u_n-u_0)^2 \, dx + \dfrac{\mu}{2}\int_{\R^N} \dfrac{(u_n-u_0)^2}{|x|} \, dx \\
	& = \dfrac{1}{2}\left\|u_n-u_0 - \sum_{k=1}^\ell w^k(\cdot - z_n^k)\right\|^2_{\mu} - \frac12 \cD(u_n-u_0) \\ &\quad + \dfrac{1}{q}\int_{\R^N} K(x)|u_n-u_0|^q \, dx 
	+ \dfrac{1}{2}\sum_{k=1}^\ell \|w^k(\cdot - z_n^k)\|_{\mu} + o(1)\\
	&= \sum_{k=1}^\ell \cE_{\per}(w^k)  + \frac12 \sum_{k=1}^\ell \cD(w^k(\cdot - z_n^k)) - \dfrac{1}{q}\int_{\R^N} K(x)|w^k(\cdot - z_n^k)|^q \, dx\\
	& \quad - \frac12 \cD(u_n-u_0) + \dfrac{1}{q}\int_{\R^N} K(x)|u_n-u_0|^q \, dx + o(1).
\end{align*}

Iterating Lemma \ref{Brezis-Lieb-D} and by Lemma \ref{BrezisLiebDecomp} we obtain
\[
	\cD(u_n-u_0) - \sum_{k=1}^\ell \cD(w^k(\cdot - z_n^k)) \to 0
\]
as $n \to +\infty$ and similarly we can prove
\[
	\int_{\R^N} K(x)|u_n-u_0|^q \, dx - \int_{\R^N} K(x)|w^k(\cdot - z_n^k)|^q \, dx \to 0
\]
as $n \to +\infty$; hence
\[
	\cE_{\per}(u_n-u_0) \to \sum_{k=1}^\ell \cE_{\per}(w^k)
\]
and the proof is complete.
\end{proof}

\section{Existence and nonexistence of ground states}

The remaining part of the proof of the existence of solutions is similar to the proof of \cite[Theorem 1.1]{BS} and \cite[Theorem 1.1]{B}, and we include it here for the reader's convenience.

\begin{proof}[Proof of Theorem \ref{ThMain1}]
Put $c_{\mathrm{per}} := \inf_{\cN_{\mathrm{per}}} \cE_{\mathrm{per}}$, where $\cE_{\mathrm{per}}$ is given by \eqref{Eper} and $\cN_{\mathrm{per}}$ is the corresponding Nehari manifold. From Theorem \ref{ThSplitting}(iii), (v) there follows that
$$
c = \lim_{n \to +\infty} \cE(u_n) = \cE(u_0) + \sum_{k=1}^\ell \cE_{\mathrm{per}} (w^k) \geq \cE(u_0) + \ell c_{\mathrm{per}}.
$$
Note that \eqref{a1} implies that $V(x) - \frac{\mu}{|x|} = V_p (x) + V_l (x) - \frac{\mu}{|x|} < V_p (x)$ for a.e. $x \in \R^N$, which gives the inequality $c_{\mathrm{per}} > c > 0$. Assume by contradiction that $u_0 = 0$. Then 
$$
c = \cE(u_0) + \sum_{k=1}^\ell \cE_{\mathrm{per}} (w^k) = \sum_{k=1}^\ell \cE_{\mathrm{per}} (w^k) \geq \ell c_{\mathrm{per}}.
$$
If $\ell \geq 1$ we obtain $c \geq \ell c_{\mathrm{per}} > \ell c$, which is a contradiction. Hence $\ell = 0$ and
$$
0 < c = \cE(u_0) = 0,
$$
a contradiction. Thus $u_0$ is a ground state solution.
\end{proof}

\begin{proof}[Proof of Theorem \ref{ThMain2}]
We assume by contradiction that $u_0$ is a ground state for $\cE$. In particular
$$
c = \inf_{\cN} \cE = \cE(u_0) > 0.
$$
The inequality \eqref{a2} implies that
$$
V(x) - \frac{\mu}{|x|} = V_p (x) + V_l (x) - \frac{\mu}{|x|} > V_p (x) \mbox{ for a.e. } x \in \R^N
$$
and therefore $c > c_{\per}$. On the other hand, fix $u \in \cN_{\per}$, where $\cE_{\mathrm{per}}$ is given by \eqref{Eper} with the corresponding Nehari manifold $\cN_{\per}$. We may choose $t_z > 0$ such that $t_z u(\cdot - z) \in \cN$ for any $z \in \Z^N$. Then
\begin{multline*}
\cE_{\per} (u) = \cE_{\per} (u (\cdot - z) ) \geq \cE_{\per} (t_z u(\cdot - z)) \\
=\cE (t_z u(\cdot - z)) - \frac12 \int_{\R^N} V_l(x) | t_z u(\cdot - z)|^2 \, dx + \frac{\mu}{2} \int_{\R^N} \frac{|t_z u(\cdot - z)|^2}{|x|} \, dx \\
\geq c - \frac12 \int_{\R^N} V_l(x) | t_z u(\cdot - z)|^2 \, dx + \frac{\mu}{2} \int_{\R^N} \frac{|t_z u(\cdot - z)|^2}{|x|} \, dx.
\end{multline*}
Note that coercivity of $\cE_{\per}$ on $\cN_{\per}$ and the inequality $\cE_{\per} (t_n (u(\cdot - z)) = \cE_{\per} (t_z u) \leq c_{\per}$ implies that $\sup_{z \in \Z^N} t_z < +\infty$. Hence
$$
\int_{\R^N} V_l(x) | t_z u(\cdot - z)|^2 \, dx = t_z^2 \int_{\R^N} V_l(x + z) u^2 \, dx \to 0\mbox{ as } |z| \to +\infty.
$$
From \cite[Lemma 2.5]{B} there follows that
$$
\int_{\R^N} \frac{|t_z u(\cdot - z)|^2}{|x|} \, dx = t_z^2 \int_{\R^N} \frac{| u(\cdot - z)|^2}{|x|} \, dx \to 0\mbox{ as } |z| \to +\infty.
$$
Hence
$$
\cE_{\per} (u) \geq c + o(1)
$$
and taking infimum over $u \in \cN_{\per}$ we see that
$$
c_{\per} = \inf_{\cN_{\per}} \cE_{\per} \geq c,
$$
which is a contradiction.
\end{proof}

\section{Compactness of ground states}

	Let $\{\mu_n\}_n \in (0,\mu^*)$ be a sequence such that $\mu_n \to 0^+$ as $n \to +\infty$ and let $\cE_n$ be the Euler functional for $\mu=\mu_n$.
We denote by $\cE_0$ and $\cN_0$ the Euler functional and the Nehari manifold for $\mu=0$ and define
\[
	c_n:=\cE_n(u_n)=\inf_{\cN_n}\cE_n, \quad c_0:=\cE_0(u_0)=\inf_{\cN_0}\cE_0,
\]
where $u_n \in \cN_n$ is the ground state solution for $\cE_n$, in particular $\cE_n ' (u_n) = 0$.

\begin{Lem}
\label{Pos_inf}
There exists  a positive radius $r>0$ such that 
\[
	\inf\limits_{n \geq 1} \inf\limits_{\|u\|_{\mu_n}=r}\cE_n(u) > 0.
\]
\end{Lem}
\begin{proof}
Fix $n\geq 1$.
Repeating the same computations as in Lemma \ref{lem_assumpt} (J1) we obtain that
\[
	\cE_n(u)=\frac{1}{2}\|u\|^2_{\mu_n} - \frac12 \cD(u) + \frac{1}{q} \int_{\R^N} K(x) |u(x)|^q \, dx \geq \frac{1}{4}\|u\|^2_{\mu_n} = \frac{1}{4}r^2 > 0
\]
for $\|u\|_{\mu_n}=r$ and for every $n \geq 1$, and properly chosen $r > 0$.
\end{proof}

\begin{Lem}
\label{ground_bounded}
The sequence $\{u_n\}_n$ is bounded in $H^{1/2}(\R^N)$.
\end{Lem}
\begin{proof}
Suppose by contradiction that $Q(u_n) \to +\infty$. Then, from Lemma \ref{norm-equiv}, we see that
\begin{align*}
Q_{\mu_n} (u_n) \geq \min \left\{ \frac12 C \left( N, \frac12 \right) - \frac{\mu_n}{C_{N,1/2,1/2}}, \essinf_{\R^N} V - m \right\} \left( [u_n]^2 + |u_n|_2^2 \right) \geq C Q(u_n),
\end{align*}
where $C > 0$ does not depend on $\mu_n$. Hence $\| u_n \|_{\mu_n} \to +\infty$. Thus
\begin{align*}
	c_0 &= \lim_{n \to +\infty}\cE_n(u_n) = \lim_{n \to +\infty}\left(\cE_n(u_n) - \dfrac{1}{q}\cE'(u_n)(u_n)\right)\\
	&=\lim_{n \to +\infty}\left[\left(\dfrac{1}{2} - \dfrac{1}{q}\right)\|u_n\|^2_{\mu_n} + \dfrac{1}{q}\int_{\R^N \times \R^N} \dfrac{F(x,u_n(x))f(y,u_n(y))u_n(y)}{|x-y|^{N-\al}} \, dx \, dy \right.\\
	& \quad\left.- \frac12 \int_{\R^N \times \R^N} \dfrac{F(x,u_n(x))F(y,u_n(y))}{|x-y|^{N-\al}} \, dx \, dy\right]\\
	&=\lim_{n \to +\infty}\left[\left(\dfrac{1}{2} - \dfrac{1}{q}\right)\|u_n\|^2_{\mu_n} + \int_{\R^N \times \R^N}\dfrac{F(x,u_n(x))\vp(y,u_n(y))}{|x-y|^{N-\al}} \, dx \, dy \right]
\end{align*}
where $\vp(y,u_n(y))=\frac{1}{q}f(y,u_n(y))u_n(y)-\frac12 F(y,u_n(y)) \geq 0$ by \eqref{AR-q}.
Hence,
\[
	c_0 \geq \lim_{n \to +\infty}\left(\dfrac{1}{2} - \dfrac{1}{q}\right)\|u_n\|^2_{\mu_n} = +\infty
\]
and we obtain a contradiction.
\end{proof}

\begin{Lem}
\label{inf_lim}
There holds
\[
	c_0=\lim_{n \to +\infty}c_n.
\]
\end{Lem}
\begin{proof}
Consider $t_n > 0$ such that $t_nu_n \in \cN_0$; note that
\begin{equation}
\label{1}
	c_n = \cE_n(u_n) \geq \cE_n(t_nu_n) = \cE_0(t_nu_n) -\dfrac{\mu_nt_n^2}{2}\int_{\R^N}\dfrac{u_n^2}{|x|}\, dx \geq c_0 - \dfrac{\mu_nt_n^2}{2}\int_{\R^N}\dfrac{u_n^2}{|x|}\, dx.
\end{equation}
Now, take $s_n>0$ such that $s_nu_0 \in \cN_n$, then
\begin{equation}
\label{2}
	c_0 = \cE_0(u_0) \geq \cE_0(s_nu_0) = \cE_n(s_nu_0) +\dfrac{\mu_ns_n^2}{2}\int_{\R^N}\dfrac{u_0^2}{|x|}\, dx \geq c_n + \dfrac{\mu_ns_n^2}{2}\int_{\R^N}\dfrac{u_0^2}{|x|}\, dx.
\end{equation}
By \eqref{1} and \eqref{2} we obtain
\[
	c_0 \geq c_n + \dfrac{\mu_ns_n^2}{2}\int_{\R^N}\dfrac{u_0^2}{|x|}\, dx \geq c_n \geq c_0 - \dfrac{\mu_nt_n^2}{2}\int_{\R^N}\dfrac{u_n^2}{|x|}\, dx,
\]
that is
\[
	c_0 - \dfrac{\mu_nt_n^2}{2}\int_{\R^N}\dfrac{u_n^2}{|x|}\, dx \leq c_n \leq c_0.
\]
Note that $\{u_n\}_n$ is bounded in view of Lemma \ref{ground_bounded}. Hence, from Lemma \ref{Hardy}, we get that $\int_{\R^N}\dfrac{u_n^2}{|x|}\, dx $ stays bounded. Hence, to complete the proof, it is sufficient to show that $\{t_n\}_n$ is bounded.

Assume, by contradiction, that $t_n \to +\infty$ then, by the fact that $t_nu_n \in \cN_0$, we have
\begin{align*}
	\cE_0'(t_nu_n)(t_nu_n) &= t_n^2 Q(u_n) - \int_{\R^N\times\R^N}\dfrac{F(x,(t_nu_n)(x))f(y,(t_nu_n)(y))(t_nu_n)(y)}{|x-y|^{N-\al}}\, dx \, dy \\
	&\quad + t_n^q\int_{\R^N}K(x)|u_n|^q \, dx = 0.
\end{align*}
Hence,
\begin{align*}
\frac{Q(u_n)}{t_n^{q-2}} = \frac12 \frac{\cD'(t_n u_n)(t_n u_n)}{t_n^q} -  \int_{\R^N} K(x) |u_n|^q \, dx.
\end{align*}
Note that, in view of Sobolev embeddings and Lemma \ref{ground_bounded}, $\int_{\R^N} K(x) |u_n|^q \, dx$ stays bounded. Moreover, Lemma \ref{ground_bounded} and $q > 2$ imply that 
$$
\frac{Q(u_n)}{t_n^{q-2}} \to 0.
$$
Hence $\frac{\cD'(t_n u_n)(t_n u_n)}{t_n^q}$ is bounded. On the other hand, \eqref{AR-q}, (F3) and Fatou's lemma imply that
\begin{align*}
\frac{\cD'(t_n u_n)(t_n u_n)}{t_n^q} &= 2 \frac{1}{t_n^q} \int_{\R^N \times \R^N} \frac{F(x, t_n u_n(x)) f(y,t_n u_n(y)) t_n u_n(y)}{|x-y|^{N-\alpha}} \, dx \, dy  \\
&\geq 2 \cdot \frac{q}{2} \frac{1}{t_n^q} \int_{\R^N \times \R^N} \frac{F(x, t_n u_n(x)) F(y,t_n u_n(y))}{|x-y|^{N-\alpha}} \, dx \, dy  \\
&= q \int_{\R^N \times \R^N} \frac{ \frac{F(x, t_n u_n(x))}{t_n^{q/2}} \frac{F(y,t_n u_n(y))}{t_n^{q/2}} }{|x-y|^{N-\alpha}} \, dx \, dy \to +\infty,
\end{align*}
which is a contradiction.
\end{proof}

\begin{proof}[Proof of Theorem \ref{ThMain3}]
Suppose that 
\[
	\lim_{n \to +\infty}\sup_{z \in \R^N}\int_{B(z,1)} |u_n|^2dx = 0.
\]
From Lion's Concentration-Compactness principle we obtain 
\[
	u_n \to 0 \text{ in } L^t(\R^N) \text{ for all } t \in \left(2,\dfrac{2N}{N-1}\right). 
\]
Recall that
\[
	0 = \cE_n'(u_n)(u_n)=\|u_n\|^2_{\mu_n}- \frac12 \cD'(u_n)(u_n) + \int_{\R^N}K(x)|u_n|^q \, dx,
\]
therefore
\[
	\|u_n\|^2_{\mu_n} = \frac12 \cD'(u_n)(u_n) - \int_{\R^N}K(x)|u_n|^q \, dx.
\]
Applying the same computation as in \eqref{D'_est} we easily get that
\[
	\cD'(u_n)(u_n) \to 0 \quad \mbox{ as } n \to +\infty
\]
and by $(K)$ we have 
\[
	\left| \int_{\R^N}K(x)|u_n|^q \, dx \right| \leq |K|_\infty |u_n|_q^q \to 0 \quad \mbox{ as } n \to +\infty.
\] 
Hence, $\|u_n\|_{\mu_n} \to 0$ as $n \to +\infty$, and since 
\begin{align*}
0 \leq \left( [u_n]^2+|u_n|_2^2 \right) &\leq \frac{Q_{\mu_n} (u_n)}{\min \left\{ \frac{1}{2} C\left( N,\frac{1}{2}\right) - \frac{\mu_n}{C_{N,1/2,1/2}}, \essinf_{\R^N} V-m \right\}} \\ 
&\to \frac{0}{\min \left\{ \frac{1}{2} C\left( N,\frac{1}{2}\right), \essinf_{\R^N} V-m \right\}} = 0
\end{align*}
we have that $u_n \to 0$ in $H^{1/2}(\R^N)$.

In view of Lemma \ref{Pos_inf} we have
\[
	\cE_n(u_n) \geq \cE_n\left(r\frac{u_n}{|u_n|}\right) > \beta > 0
\]
for some constant $\beta$, and by Lemma \ref{ground_bounded}
\begin{equation*}
	\limsup_{n \to +\infty} \cE_n(u_n) = \limsup_{n \to +\infty} \left( -\frac12 \cD(u_n) \right) \leq 0
\end{equation*}
so we reach a contradiction.

Hence, there is a sequence $\{z_n\}_n \subset \Z^N$ such that 
\[
	\liminf_{n \to +\infty} \int_{B\left(z_n,1+\sqrt{N}\right)}|u_n|^2dx >0.
\]
In view of Lemma \ref{ground_bounded}, there is $u \in H^{1/2}(\R^N) \setminus \{0\}$ such that
\begin{align*}
	&u_n(\cdot + z_n) \weakto u \text{ in } H^{1/2}(\R^N), \notag\\
	&u_n(\cdot + z_n) \to u \text{ in } L^2_{\mathrm{loc}}(\R^N), \\
	&u_n(x + z_n) \to u(x) \text{ for a.e. } x \in \R^N.\notag
\end{align*}
Let $w_n=u_n(\cdot + z_n)$ and fix any $\vp \in \cC^{\infty}_0(\R^N)$. Observe that
\begin{align*}
	\cE'_0(w_n)(\vp) &= \cE'_n(u_n)(\vp(\cdot - z_n)) + \mu_n \int_{\R^N}\dfrac{u_n \vp(\cdot-z_n)}{|x|} \, dx\\
	&=\mu_n \int_{\R^N}\dfrac{u_n \vp(\cdot-z_n)}{|x|} \, dx.
\end{align*}
By \cite[Lemma 2.5]{B} and H\"older's inequality we have that
\[
	\int_{\R^N}\dfrac{|u_n(x)||\vp(x-z_n)|}{|x|} \, dx \to 0 \quad \mbox{ as } n \to +\infty,
\]
hence $\cE'_0(w_n)(\vp) \to 0$. From Corollary \ref{weak-conv-E} there follows that 
\[
	\cE_0'(w_n)(\vp) \to \cE_0'(u)(\vp)
\]
thus $u$ is a nontrivial critical point of $\cE_0$. In particular, $u \in \cN_0$. By \eqref{AR-q}, Lemma \ref{inf_lim} and Fatou's Lemma we have

\begin{align*}
c_0 &= \liminf_{n \to +\infty}\cE_n(u_n) = \liminf_{n \to +\infty}\left(\cE_n(u_n) - \dfrac{1}{q}\cE_n'(u_n)(u_n)\right)\\
	&= \liminf_{n \to +\infty} \left[ \left(\frac12 - \frac1q \right) Q(u_n) + \left( \frac1q - \frac12 \right) \mu_n \int_{\R^N} \frac{u_n^2}{|x|} \, dx \right. \\
	&\quad \left. + \frac1q \int_{\R^N \times \R^N} \frac{F(x,u_n(x)) \left( f(y,u_n(y)) u_n(y) - \frac{q}{2} F(y,u_n(y)) \right)}{|x-y|^{N-\alpha}} \, dx \, dy \right] \\
	&= \liminf_{n \to +\infty} \left[ \left(\frac12 - \frac1q \right) Q(w_n) + \left( \frac1q - \frac12 \right) \mu_n \int_{\R^N} \frac{u_n^2}{|x|} \, dx \right. \\
	&\quad \left. + \frac1q \int_{\R^N \times \R^N} \frac{F(x,w_n(x)) \left( f(y,w_n(y)) w_n(y) - \frac{q}{2} F(y,w_n(y)) \right)}{|x-y|^{N-\alpha}} \, dx \, dy \right] \\
	&\geq \left( \frac12 - \frac1q \right) Q(u) + \frac1q \int_{\R^N \times \R^N} \frac{F(x,u(x)) \left( f(y,u(y)) u(y) - \frac{q}{2} F(y,u(y)) \right)}{|x-y|^{N-\alpha}} \, dx \, dy \\
	&= \cE_0(u) - \frac1q \cE_0'(u)(u) = \cE_0(u) \geq c_0,
\end{align*}
where we used the weak lower semicontinuity of the norm $Q(\cdot)$ and the fact that $\mu_n \int_{\R^N} \frac{u_n^2}{|x|} \, dx \to 0$ from the Hardy inequality. Hence, $\cE_0(u) = c$ and $u \in H^{1/2}(\R^N)$ is a ground state solution for $\cE_0$.

\end{proof}

\section*{Acknowledgements}

Federico Bernini and Simone Secchi are members of INdAM-GNAMPA. Bartosz Bieganowski was partially supported by the National Science Centre, Poland (Grant No. 2017/25/N/ST1/00531).

\end{document}